\documentclass[12 pt]{article}%
\usepackage{amsmath, amsfonts, amsthm, color,latexsym,breqn}
\usepackage{amssymb}
\usepackage{color, soul}
\usepackage{tikz}
\usepackage[all]{xy}
\usepackage[T1]{fontenc}
\usepackage[utf8]{inputenc}
\usepackage[normalem]{ulem}
\usepackage{graphicx}%
\providecommand{\U}[1]{\protect\rule{.1in}{.1in}}
\allowdisplaybreaks[4]
\newtheorem{theorem}{Theorem}[section]
\newtheorem{proposition}[theorem]{Proposition}
\newtheorem{corollary}[theorem]{Corollary}
\newtheorem{example}[theorem]{Example}

\newtheorem{remark}[theorem]{Remark}

\newtheorem{lemma}[theorem]{Lemma}
\newtheorem{final remark}[theorem]{Final Remark}
\newtheorem{definition}[theorem]{Definition}
\allowdisplaybreaks[4]

\begin{document}

\title{Summing multilinear operators and sequence classes}
\author{Geraldo Botelho\thanks{Supported by CNPq Grant
304262/2018-8 and Fapemig Grant PPM-00450-17.}\,\, and  Davidson Freitas\thanks{Supported by a CNPq scholarship.\newline 2020 Mathematics Subject Classification: 46G25, 47B10, 47H60, 47L22.\newline Keywords: Banach spaces, summing multilinear operators, Banach ideals, sequence classes.
}}
\date{}
\maketitle

\begin{abstract}
We construct a general framework that generates classes of multilinear operators between Banach spaces which encompasses, as particular cases, the several classes of summing type multilinear operators that have been studied individually in the literature. Summing operators by blocks in the isotropic and anisotropic cases are taken into account. The classes we create are shown to be Banach ideals of multilinear operators and applications to coherence and coincidence theorems are provided.
\end{abstract}

\section{Introduction}
\label{seccao-2.1-introducao}

A. Pietsch systematized the theory of ideals of linear operators (operator ideals) 
in \cite{pietsch_operator_1978} and, shortly after, sketched the theory of ideals of multilinear operators in 
\cite{pietsch_ideals_1984}. This field has proved to be quite fruitful and a large number of contributions have appeared in the last decades, 
 see, e.g., \cite{achour_factorization_2014, albuquerque_summability_2018, albuquerque_note_2019, albuquerque_absolutely_2015, araujo_optimal_2015, bayart_multiple_2018,botelho_type_2016, botelho_transformation_2017, botelho_operator_2017, campos_cohen_2014,dimant_diagonal_2019,nunez-alarcon_sharp_2019,ribeiro_generalized_2019}. Several authors have devoted their efforts to study multilinear generalizations of the ideals of absolutely $(q;p)$-summing linear operators to the multilinear context. Several promising classes of multilinear operators have arisen from these efforts, for example, the class of absolutely summing multilinear operators, initiated in \cite{alencar_some_1989, pietsch_ideals_1984}, and the class of multiple summing operators introduced in \cite{matos_fully_2003} and, independently, in \cite{bombal_multilinear_2004}. Also, other operator ideals defined, or characterized, by the transformation of vector-valued sequences, were transposed to the multilinear case, and, again, in different manners. In this paper we construct an abstract framework that recover as particular instances all these classes of multilinear operators tha have been individually studied in the literature. 

Throughout the paper, $d\geq 1$ is a natural number, 
$\mathcal{L}(E_1,\ldots,E_d;F)$ denotes the Banach space of continuous $d$-linear operators quando 
from $E_1\times \cdots \times E_d$ to $F$, where $E_1,\ldots, E_d,F$ will always be Banach spaces over $\mathbb{K}=\mathbb{R}$ or $\mathbb{C}$, endowed with the usual operator norm. The topological dual of the Banach space $E$ is denoed by $E^*$ and the closed unit ball of $E$ by $B_E$. For the general theory of multilinear operators we refer to 
\cite{dineen_complex_1999, mujica_complex_1986}. 

According to the first transposition of the absolutely summing operators to the multilinear setting, which goes back to
\cite{pietsch_ideals_1984}, an operator $A\in\mathcal{L}(E_1,\ldots,E_d;F)$ is said to be absolutely $(q;p_1,\ldots,p_d)$-summing, $\frac{1}{q}\leq \frac{1}{p_1}+\cdots+\frac{1}{p_d}$  and $q,p_1,\ldots,p_d\in [1,\infty)$, if there exists a constant $C>0$ such that
\begin{align}
    \left(\sum_{j=1}^k\| A(x_j^1,\ldots,x_j^d)\|^q\right)^{\frac{1}{q}}\leq C\cdot \prod_{i=1}^d\sup_{\varphi_i\in B_{E_i^*}} \left(\sum_{j=1}^k\| \varphi_i(x_j^i)\|^{p_i}\right)^{\frac{1}{p_i}},
    \label{eq-2.1}
\end{align}
for every $k\in\mathbb{N}$ and all  $x_j^i\in E_i$, $j=1,\ldots,k$, $i=1,\ldots,d$.

Going deeper into the multilinear aspects of the theory, 
Matos \cite{matos_fully_2003} 
and Bombal, Villanueva and Pérez-García \cite{ bombal_multilinear_2004}  define a multilinear operator $A\in\mathcal{L}(E_1,\ldots,E_d;F)$ to be multiple $(q;p_1,\ldots,p_d)$-summing, 
$q\geq p_r$, $r=1,\ldots,d$, if there exists a constant $C\geq 0$ such that
    \begin{align}
    \left(\sum_{j_1,\ldots,j_d=1}^k\| A(x_{j_1}^1,\ldots,x_{j_d}^d)\|^q\right)^{\frac{1}{q}}\leq C\cdot \prod_{i=1}^d\sup_{\varphi_i\in B_{E_i^*}} \left(\sum_{j=1}^k\| \varphi_i(x_j^i)\|^{p_i}\right)^{\frac{1}{p_i}},
    \label{operadores_multiplo_matos}
\end{align}
for every $k\in\mathbb{N}$ and all  $x_j^i\in E_i$, $j=1,\ldots,k$, $i=1,\ldots,d$.

We are interested in the following four directions in which the theory has been developed nowadays:


(i) An operator $A$ is absolutely $(q;p_1, \ldots, p_d)$-summing if and only if, for all $E_i$-valued weakly $p_i$-summable sequences $(x_j^i)_{j=1}^{\infty}$, $i = 1, \ldots, d$,  the sequence 
$\left(A(x_j^1,\ldots,x_j^d)\right)_{j=1}^{\infty}$ is absolutely $q$-summable in $F$. 
Several new classes of multilinear operators appeared considering other types of summability of sequences, such as: $p$-dominated operators \cite{blasco_coincidence_2016, popa_composition_2014}, almost summing operators \cite{botelho_almost_2001,pellegrino_almost_2012, pellegrino_fully_2006}; weakly $(q;p_1, \ldots, p_d)$-summing operators \cite{blasco_coincidence_2016,botelho_transformation_2017,kim_multiple_2007,popa_mixing_2012}, 
   strongly Cohen $p$-summing opertors \cite{achour_cohen_2007,campos_cohen_2014,mezrag_inclusion_2009}. A unification of all these types of operators was proposed in \cite{botelho_transformation_2017} by means of the concept of sequence classes. 

(ii) The definition (\ref{operadores_multiplo_matos}) of multiple summing operators takes into account the sum over all indices independently, that is, all multi-indices $(n_1, \ldots, n_d) \in \mathbb{N}^d$, whereas the definition (\ref{eq-2.1}) considers only the diagonal $D(\mathbb{N}^d) := \{(j,\ldots,j) : j \in \mathbb{N}\}$ of $ \mathbb{N}^d$. The sum in all/several indices has been deeply explored, including in the study of classical inequalities and applications to other areas (see, e.g., 
\cite{bayart_multiple_2018,defant_coordinatewise_2010, montanaro_applications_2012, perez-garcia_unbounded_2008}).

(iii) Another direction is the consideration of sums over sets of indices other than the diagonal or the whole $\mathbb{N}^d$, including arbitrary subsets of $\mathbb{N}^d$, usually referred to as blocks. For some recent works in this direction, see 
\cite{albuquerque_summability_2018,albuquerque_summability_2019, araujo_classical_2016, bayart_coincidence_2020, botelho_summability_2008,zalduendo_estimate_1993}. 

(iv) The replacement of the sum in the left-hand side of (\ref{operadores_multiplo_matos}) with iterated sums of the type
\begin{align*}
    \left(\sum_{j_1=1}^{\infty}\left(\cdots\left(\sum_{j_d=1}^{\infty}\| A(x_{j_1}^1,\ldots,x_{j_d}^d)\|^{q_d}\right)^{\frac{q_{d-1}}{q_d}}\cdots\right)^{\frac{q_1}{q_2}}\right)^{\frac{1}{q_1}}
\end{align*}
has led to very interesting results, including applications to the study of classical inequalities, such as Bohnenblust-Hille's inequality and Hardy-Littlewood's inequality, see, e.g.,
\cite{albuquerque_summability_2018, albuquerque_note_2019, albuquerque_holders_2017a, albuquerque_summability_2019, albuquerque_sharp_2014, albuquerque_optimal_2016,albuquerque_applications_2017, aron_optimal_2017, bayart_multiple_2018, nunez-alarcon_sharp_2019}. This is the anisotropic case, whereas the other cases are the isotropic ones.

A unification of all these directions was attempted by the authors in \cite{botelho_summing_2020}, where a number of classes already considered in the literature are recovered as particular instances of an abstract scheme. The fact that some already studied classes, including some important ones, are not encompassed by the construction in \cite{botelho_summing_2020} impelled us to construct a new environment, which we present in this paper, encompassing, as particular instances, all classes of  summing-type multilinear operators defined, or characterized, by the transformation of sequences that have been investigated thus far.

The construction and the proof that the classes we create are Banach ideals of multilinear operators are presented in Section 2. In Section 3 we show that our construction recovers the classes studied in the literature, including some classes not recovered in \cite{botelho_summing_2020}. In Section 4 we give applications to the theory of coherent ideals and coincidence theorems.


\section{The construction}
\label{sec-2.2-construcao}

Throughout the paper, $d$ and $k$ are natural numbers with  $1\leq k\leq d$, $(j^1,\ldots,j^d)$ is an element of $\mathbb{N}^d$, $\mathcal{I} = \{I_1,\ldots,I_k\}$ is a partition of $\{1,\ldots,d\}$, that is, an ordered collection of pairwise disjoint subsets of $\{1,\ldots,d\}$ whose unioin is $\{1,\ldots,d\}$. For $x \in E$, we write $x\cdot e_j= (0,\ldots,0,x,0,\ldots) \in E^{\mathbb{N}}$  and $x *e_j = (0,\ldots,0,x,0,\ldots,0) \in E^d $, where $x$ is placed at the $j$-th coordinate. 

\subsection{Blocks in $\mathbb{N}^d$}
\label{subsec_blocos}

\begin{definition}\rm\label{definicao-bloco}
     Given $d$ sequences of natural numbers $(j_n^r)_{n=1}^{\infty}$, $r=1,\ldots,d$, such that the correspondence
    \begin{align}
		(n_1,\ldots , n_k)\in \mathbb{N}^k \mapsto \sum_{s=1}^k \sum_{r\in I_s} j_{n_s}^r *e_r
		\label{injetividade_bloco}
    \end{align}
    is injective, the {\it block of $\mathbb{N}^d$ associated to the partition $\mathcal{I} = \{I_1,\ldots,I_k\}$ and to the sequences $(j_n^r)_{n=1}^{\infty}$, $r=1,\ldots,d$,} is the subset of $\mathbb{N}^d$ defined by
    \begin{align}
	    B_{\mathcal{I}}=\left\{\sum_{s=1}^k \sum_{r\in I_s} j_{n_s}^r *e_r \in \mathbb{N}^d: n_1,\ldots , n_k\in \mathbb{N} \right\}.
	    \label{eq-1}
    \end{align}
\end{definition}

The understanding of the expression $\sum\limits_{s=1}^k \sum\limits_{r\in I_s} j_{n_s}^r *e_r$ is crucial in our construction. 
This finite sum represents an element $(j^1_{m_1},\ldots,j^d_{m_d}) \in \mathbb{N}^d$, where $m_n = m_l=n_s$ for some $s\in \{1,\ldots, k\}$ whenever $n,l\in I_s$, that is, each natural $j^r_{m_r}$ is indexed by $m_r = n_s$ for every $r\in I_s$. Obviously the injectivity guarantees that each position of the block is achieved exactly once. The idea is not to check the injectivity of a given expression (\ref{injetividade_bloco}), but, given $B  \subseteq \mathbb{N}^d$, to find a partition of $\{1, \ldots, d\}$ and $d$ sequences of naturals  such that   $B$ is a block of $\mathbb{N}^d$ associated to the partition and to the sequences. 

Let us see some blocks we shall use later. 

\begin{example}\label{exemplo_blocos_a}\rm
    Given an arbitrary partition $\mathcal{I} = \{I_1,\ldots,I_k\}$ of $\{1,\ldots,d\}$, defining $(j_n^1)_{n=1}^{\infty} = \cdots = (j_n^d)_{n=1}^{\infty} = (n)_{n=1}^{\infty}$, the correspondence 
    \begin{align*}
		(n_1,\ldots , n_k)\in \mathbb{N}^k \mapsto \sum_{s=1}^k \sum_{r\in I_s} j_{n_s}^r *e_r &= \sum_{s=1}^k \sum_{r\in I_s} n_s *e_r= \sum_{r\in I_1} n_1 *e_r +\cdots + \sum_{r\in I_k} n_k *e_r,
    \end{align*}
    is injective: if
    \begin{align*}
        \sum_{r\in I_1} n_1 *e_r +\cdots + \sum_{r\in I_k} n_k *e_r = \sum_{r\in I_1} m_1 *e_r +\cdots + \sum_{r\in I_k} m_k *e_r,
    \end{align*}
then $n_1 = m_1$ for every $r\in I_1, \ldots$, $n_k = m_k$ for every $r\in I_k$, that is, $(n_1,\ldots,n_k) = (m_1,\ldots,m_k)$. 

In the case of an arbitrary partition, this block shall be used in Section \ref{secao_parcialmente_multiplo_somante}. It shall also be used for the following two specific partitions. 
    Denoting by $\mathcal{I}_t := \{1,\ldots,d\}$ the trivial partition ($k = 1$), we get the diagonal  $$B_{\mathcal{I}_t} : = D(\mathbb{N}^d)= \{(j,\ldots,j)\in\mathbb{N}^d\}.$$
And denoting by $\mathcal{I}_d := \{\{r\}:r=1,\ldots,d\}$ the discrete partition ($k = d$), the corresponding block is $B_{\mathcal{I}_d} = \mathbb{N}^d$.
\end{example}

%

\begin{example}\label{exemplo_blocos_c}\rm Let us see that every infinite subset of $\mathbb{N}^d$ is a block associated to the trivial partition ${\cal I}_t$ in the sense of Definition \ref{definicao-bloco}. 
Given $B \subset \mathbb{N}^d$ infinite, choose a bijection $\sigma \colon \mathbb{N}\longrightarrow B$. For each $(j^1,\ldots,j^d)\in B$ let $n\in\mathbb{N}$ be such that $(j^1,\ldots,j^d) = \sigma(n)$. For $r\in\{1,\ldots,d\}$ and $n\in\mathbb{N}$, call $j_n^r$ the $r$-th coordinate of $\sigma(n)$. So, the correspondence (\ref{injetividade_bloco})
$$n\in\mathbb{N}\mapsto \sum_{r=1}^d j_n^r *e_r = j_n^1 *e_1 + \cdots +j_n^d *e_d= (j_n^1,\ldots,j_n^d) = \sigma(n)\in B,$$
is the bijection $\sigma$ itself and $B_{\mathcal{I}_t} = \{ \sigma(n):n\in\mathbb{N}^d\} = B$.

\end{example}

By the example above we could restrict ourselves to the trivial partition. We shall not do that because sometimes it is more convenient to recognized a certain subset of 
$\mathbb{N}^d$ as a block associate do some other partition. It is not difficult to see that not every infinite subset of $\mathbb{N}^d$ is a block associated to the discrete partition.  

%


From now on, we fix the natural numbers $1\leq k\leq d$, a partition  $\mathcal{I} = \{I_1,\ldots,I_k\}$ of  $\{1,\ldots,d\}$, the sequences of natural numbers $(j_n^r)_{n=1}^{\infty}$, $r=1,\ldots,d$, fulfilling the conditions of Definition \ref{definicao-bloco} and, consequently, the block $B_{\mathcal{I}}$  associated to this partition and to these sequences.

\subsection{$(B_{\mathcal{I}};X_1,\ldots,X_d;Y_1,\ldots,Y_k)$-summing operators}
By $c_{00}(E)$ and $\ell_{\infty}(E)$ we denote the spaces of eventually null and bounded $E$-valued sequences, respectively. The symbol $E\stackrel{1}{\hookrightarrow} F$ means that $E$ is a linear subspace of $F$ and  com $\| x\|_F\leq \| x\|_E$ for every $x\in E$.

\begin{definition}\rm \cite[Definition 2.1]{botelho_transformation_2017} \label{definicao_classe_sequências}
	A \textit{sequence class} is a rule $X$ that assigns to each Banach space $E$ a Banach space $X(E)$ of $E$-valued sequences, with the usual sequence operations, such that 
	$c_{00}(E)\subset X(E)\stackrel{1}{\hookrightarrow} \ell_{\infty}(E)$ for every $E$ and  $\| e_j\|_{X(\mathbb{K})} = 1$ for every $j\in\mathbb{N}$.  
    The sequence class $X$ is \textit{linearly stable} if, regardless of the Banach spaces $E$ and $F$, $u\in\mathcal{L}(E;F)$ and $(x_j)_{j=1}^{\infty}\in X(E)$ it holds
		\begin{align*}
		    (u(x_j))_{j=1}^{\infty}\in X(F) \mbox{ and } \| (u(x_j))_{j=1}^{\infty}\|_{X(F)}\leq \| u\|\cdot\| (x_j)_{j=1}^{\infty} \|_{X(E)}.
		\end{align*}
In this case, the induced map
    $$\widehat{u}\colon  X(E)\longrightarrow X(F)~,~ \widehat{u}\left((x_j)_{j=1}^{\infty}\right) = (u(x_j))_{j=1}^{\infty},$$
is a well defined bounded linear operator.
\end{definition}

 Plenty of examples can be found in \cite{botelho_transformation_2017, botelho_operator_2017, joedsonjamilson, ribeiro_generalized_2019, ribeiro_absolutely_2021}; we mention just a few: for $p \geq 1$,  the class $\ell_p(\cdot)$ of absolutely $p$-summable sequences, the class $\ell_p^w(\cdot)$ of weakly $p$-summable sequences, the class $c_0(\cdot)$ of norm null sequences, the class $\ell_{\infty}(\cdot)$ of bounded sequences, the class $\ell_p^u(\cdot)$ of unconditionally $p$-summable sequences, the class $\ell_p\langle\cdot\rangle$ of Cohen strongly $p$-summable sequences. In this section, $X_1, \ldots, X_d, Y_1,\ldots, Y_k$ denote arbitrary sequence classes.

Given a Banach space $F$ and sequence classes $Y_1$ and $Y_2$, $Y_1(Y_2(F))$ is space of sequences $((x_{j_1,j_2})_{j_2=1}^{\infty})_{j_1=1}^{\infty}\in (Y_2(F))^{\mathbb{N}}$ such that $((x_{j_1,j_2})_{j_2=1}^{\infty})_{j_1=1}^{\infty}\in Y_1(Y_2(F))$. Iterating we construct the space $Y_1(\cdots Y_k(F) \cdots )$. 
The next definition is inspired in \cite{albuquerque_summability_2018, albuquerque_anisotropic_2018, araujo_classical_2016, araujo_optimal_2015, botelho_summing_2020}.

\begin{definition}\rm \label{definicao-operadores somantes}
	A continuous $d$-linear operator $A\in\mathcal{L}(E_1,\ldots , E_d;F)$ is $(B_{\mathcal{I}};X_1,\ldots, X_d;$ $Y_1,\ldots,Y_k)$-summing if
	\begin{align}
		\left( \cdots \left( A\left( \sum_{s=1}^k \sum_{r\in I_s} x_{j_{n_s}^r}^r *e_r \right) \right)_{n_k=1}^{\infty} \cdots\right)_{n_1=1}^{\infty} \in Y_1(\cdots Y_k(F) \cdots ),
		\label{expressao_def_operadores_somantes}
	\end{align}
whenever $(x_j^r)_{j=1}^{\infty}\in X_r(E_r)$, $r=1,\ldots,d$.
\end{definition}

The case where $\cal I$ is the trivial partition ${\cal I}_t$, that is, $k=1$ and $\mathcal{I}_t = \{1,\ldots,d\}$, is called the {\it isotropic case}. The other cases, where we have mixed sums, are called {\it anisotropic cases}.

In the $d$-tuple $\sum\limits_{s=1}^k \sum\limits_{r\in I_s} x_{j_{n_s}^r}^r *e_r\in E_1\times \cdots \times E_d$, for each $s\in \{1,\ldots,k\}$ the elements in the $r$-th coordinate, $r\in I_s$, are indexed by $j^r_{n_s}$ which are indexed by $n_s$, that is, the indices of the $r$-th coordinates are indexed by the same $n_s$. Next we see a few illustrative examples, in which the sequences $(x_j^r)_{j=1}^{\infty}\in X_r(E_r)$, $r=1,\ldots,d$, are given.

\begin{example}\rm \label{exemplo_expressao_a}
    Considering an arbitrary partition $\mathcal{I}= \{I_1,\ldots,I_k\}$ and the sequences $(j_n^1)_{n=1}^{\infty} = \cdots = (j_n^d)_{j=1}^{\infty} = (n)_{n=1}^{\infty}$, we have 
    \begin{align*}
        \sum\limits_{s=1}^k \sum\limits_{r\in I_s} x_{j_{n_s}^r}^r *e_r &= \sum\limits_{r\in I_1} x_{n_1}^r *e_r+\cdots + \sum\limits_{r\in I_k} x_{n_k}^r *e_r,
    \end{align*}
that is, all vectors in the $r$-th coordinate, $r\in I_s$, $s=1,\ldots,k$, are indexed by the same $n_s$. Take the trivial partition $\mathcal{I}_t$. For all $n \in \mathbb{N}$ and $x_n^r\in E_r$, $r=1,\ldots,d$,
    \begin{align*}
        \sum\limits_{s=1}^k \sum\limits_{r\in I_s} x_{j_{n_s}^r}^r *e_r &= \sum\limits_{r\in \{1,\ldots,d\}} x_{n}^r *e_r = x_n^1 *e_1 + \cdots + x_n^d *e_d \\
        &= (x_n^1,0,\ldots,0) + \cdots + (0,\ldots, 0 ,x_n^d)= (x_n^1,\ldots,x_n^d).
    \end{align*}
    Take now the discrete partition $\mathcal{I}_d  = \{I_1,\ldots, I_d\} = \{\{1\}, \ldots, \{d\}\}$. For all  $n_r\in \mathbb{N}$ and $x_{n_r}^r\in E_r$, $r=1,\ldots,d$,
    \begin{align*}
        \sum\limits_{s=1}^k \sum\limits_{r\in I_s} x_{j_{n_s}^r}^r *e_r &= \sum_{s=1}^d\sum\limits_{r\in I_s} x_{n_s}^r *e_r = \sum\limits_{r\in I_1} x_{n_1}^r *e_r + \cdots + \sum\limits_{r\in I_d} x_{n_d}^r *e_r \\
        &= \sum\limits_{r\in \{1\}} x_{n_1}^r *e_r + \cdots + \sum\limits_{r\in \{d\}} x_{n_d}^r *e_r=  x_{n_1}^1 *e_1 + \cdots  +x_{n_d}^d *e_d\\
        &= (x_{n_1}^1,0,\ldots,0) + \cdots + (0,\ldots, 0 ,x_{n_d}^d)= (x_{n_1}^1,\ldots ,x_{n_d}^d).
    \end{align*}
    \end{example}



By $\mathcal{L}_{X_1,\ldots,X_d;Y_1,\ldots,Y_k}^{B_{\mathcal{I}}}(E_1,\ldots,E_d;F)$ we mean the class of all operators in $\mathcal{L}(E_1,\ldots,E_d;F)$ that are  $(B_{\mathcal{I}};X_1,\ldots,X_d;Y_1,\ldots,Y_k)$-summing. Some simplifications will be adopted: \\
$\bullet$ $\mathcal{L}_{X_1,\ldots,X_d;Y_1,\ldots,Y_k}^{B_{\mathcal{I}}}(E_1,\ldots,E_d)$ if $F=\mathbb{K}$.\\
$\bullet$ $\mathcal{L}_{X_1,\ldots,X_d;Y_1,\ldots,Y_k}^{B_{\mathcal{I}}}(^dE;F)$ if $E_1 = \cdots = E_d=E$.\\
$\bullet$ $\mathcal{L}_{^dX;Y_1,\ldots,Y_k}^{B_{\mathcal{I}}}(E_1,\ldots,E_d;F)$ if $X_1=\cdots=X_d=X$.\\
$\bullet$ $\mathcal{L}_{X_1,\ldots,X_d;^kY}^{B_{\mathcal{I}}}(E_1,\ldots,E_d;F)$ if $Y_1=\cdots=Y_k=Y$.

Obvious combinations of these simplifications will be used too. 

\begin{remark}\rm In the isotropic case, where the partition is the trivial one, the sequences of natural numbers that define the block depend on a chosen bijection $\sigma \colon \mathbb{N}\longrightarrow B$ (cf. Example 
\ref{exemplo_blocos_c}). 
In this case, in (\ref{expressao_def_operadores_somantes}) we have a sequence whose terms are $A(x_{j_1}^1, \ldots, x_{j_d}^d)$ with $(j_1, \ldots, j_d) \in B$, but the order these terms appear in the sequence depend on 
$\sigma$. This dependence no longer exists 
if the sequence class $Y$ is  {\it symmetric} in the sense that, for every Banach space $F$ and any sequence $(y_j)_{j=1}^{\infty}$ in $F$, it holds $(y_j)_{j=1}^{\infty}\in Y(F)$ if an only if $(y_{s(j)})_{j=1}^{\infty}\in Y(F)$ for every permutation $s \colon \mathbb{N}\longrightarrow\mathbb{N}$ and, in this case, $\| (y_j)_{j=1}^{\infty}\|_{Y(F)} = \| (y_{s(j)})_{j=1}^{\infty}\|_{Y(F)}$. For instance, the sequence classes $c_0(\cdot), \ell_p(\cdot)$ and $\ell_p^w(\cdot)$, $1 \leq p \leq \infty$, are symmetric. Symmetric spaces of scalar-valued sequences are treated in \cite{mathnotes}. \label{obs-2}
\end{remark}

%


\begin{proposition} \label{proposicao_operador_induzido}
     An operator $A\in \mathcal{L}(E_1,\ldots,E_d;F)$ is $(B_{\mathcal{I}};X_1,\ldots, X_d;Y_1,\ldots,Y_k)$-sum-ming if and only if the induced map $\widehat{A}_{B_{\mathcal{I}}} \colon X_1(E_1)\times\cdots \times X_d(E_d)\longrightarrow$ $ Y_1(\cdots Y_k(F) \cdots )$,
            \begin{align*}
                \widehat{A}_{B_{\mathcal{I}}} \left( (x_j^1)_{j=1}^{\infty},\ldots , (x_j^d)_{j=1}^{\infty} \right) =\left( \cdots \left( A\left( \sum_{s=1}^k \sum_{r\in I_s} x_{j_{n_s}^r}^r* e_r \right) \right)_{n_k=1}^{\infty} \cdots\right)_{n_1=1}^{\infty},
            \end{align*}
is a well defined continuous $d$-linear operator.
\end{proposition}
\begin{proof} To prove the nontrivial implication, assume that $A$ is $(B_{\mathcal{I}};X_1,\ldots, X_d;$ $Y_1,\ldots,Y_k)$-summing. Of course the induced map $\widehat{A}_{B_{\mathcal{I}}}$ is well defined. We omit the proof of its $d$-linearity.  
We prove its continuity by applying the closed graph theorem for multilinear operators 
(see \cite{fernandez_closed_1996}). Let $\overline{x}_l^r = (x_{l,j}^r)_{j=1}^{\infty}$ and $\overline{x}^r = (x_j^r)_{j=1}^{\infty}$ be sequences in $X(E_r)$, $r\in \{ 1,\ldots, d\}$ and $l\in\mathbb{N}$, such that $\overline{x}_l^r\stackrel{l}{\longrightarrow} \overline{x}^r$ in $X_r(E_r)$, $r=1,\ldots, d$, and
	\begin{align}
		\widehat{A}_{B_{\mathcal{I}}}(\overline{x}_l^1,\ldots, \overline{x}_l^d) \stackrel{l}{\longrightarrow} \left( \cdots \left( z_{n_1,\ldots,n_k}\right)_{n_k=1}^{\infty} \cdots\right)_{n_1=1}^{\infty} \mbox{ in } Y_1(\cdots Y_k(F)\cdots).
		\label{eq3}
	\end{align}
The condition $X_r(\cdot)\stackrel{1}{\hookrightarrow} \ell_{\infty}(\cdot)$ guarantees that 
$x_{l,j}^r\stackrel{l}{\longrightarrow} x_j^r$ for every $j\in\mathbb{N}$. Since $A$ is continuous, $A\left( x_{l,j^1}^1,\ldots,x_{l,j^d}^d\right) \stackrel{l}{\longrightarrow} A\left( x_{j^1}^1,\ldots,x_{j^d}^d\right)$ in $F$ for any $j^1,\ldots,j^d\in\mathbb{N}$. In particular, as $j_{n_s}^r\in \mathbb{N}$ for all $s\in \{1,\ldots,k\}$ and $r \in I_s$,
	\begin{align}
		A\left( \sum_{s=1}^k\sum_{r\in I_s} x_{l,j_{n_s}^r}^r * e_r\right) \stackrel{l}{\longrightarrow}A\left( \sum_{s=1}^k\sum_{r\in I_s} x_{j_{n_s}^r}^r * e_r\right)  \mbox{ in } F	,
		\label{eq4}
	\end{align}
	for all $n_1,\ldots,n_k\in\mathbb{N}$. And since $\widehat{A}_{B_{\mathcal{I}}}(\overline{x}_l^1,\ldots, \overline{x}_l^d)  =$ $$\widehat{A}_{B_{\mathcal{I}}}((x_{l,j}^1)_{j=1}^{\infty},\ldots, (x_{l,j}^d)_{j=1}^{\infty}) = \left(\cdots\left( A\left( \sum_{s=1}^k\sum_{r\in I_s} x_{l,j_{n_s}^r}^r * e_r\right) \right)_{n_k=1}^{\infty}\cdots\right)_{n_1=1}^{\infty},$$
from (\ref{eq3}) it follows that
	\begin{align*}
		 \left(\cdots\left( A\left( \sum_{s=1}^k\sum_{r\in I_s} x_{l,j_{n_s}^r}^r * e_r\right) \right)_{n_k=1}^{\infty}\cdots\right)_{n_1=1}^{\infty} &\stackrel{l}{\longrightarrow} \left( \cdots \left( z_{n_1,\ldots,n_k}\right)_{n_k=1}^{\infty} \cdots\right)_{n_1=1}^{\infty}
	\end{align*}
 in $Y_1(\cdots Y_k(F)\cdots)$. Therefore, for all $n_1,\ldots,n_k\in\mathbb{N}$, 
	\begin{align}
		 A\left( \sum_{s=1}^k\sum_{r\in I_s} x_{l,j_{n_s}^r}^r * e_r\right)  \stackrel{l}{\longrightarrow} z_{n_1,\ldots,n_k} \mbox{ in } F.
		 \label{eq5}
	\end{align}
	From (\ref{eq4}) and (\ref{eq5}) we conclude that 
$
 \widehat{A}_{B_{\mathcal{I}}}(\overline{x}^1,\ldots,\overline{x}^d) = \left( \cdots \left( z_{n_1,\ldots,n_k}\right)_{n_k=1}^{\infty} \cdots\right)_{n_1=1}^{\infty}$,
proving that $\widehat{A}_{B_{\mathcal{I}}}$ is continuous.
\end{proof}


 For an operator $A\in\mathcal{L}_{X_1,\ldots, X_d; Y_1,\ldots,Y_k}^{B_{\mathcal{I}}}(E_1,\ldots , E_d;F)$, we define
    \begin{align}
        \| A\|_{B_{\mathcal{I}}; X_1,\ldots, X_d; Y_1,\ldots,Y_k} = \| \widehat{A}_{B_{\mathcal{I}}} \|.
        \label{eq_def_norma}
    \end{align}

Our next purpose is to prove that $\mathcal{L}_{X_1,\ldots, X_d; Y_1,\ldots,Y_k}^{B_{\mathcal{I}}}$ endowed with $\|\cdot\|_{B_{\mathcal{I}}; X_1,\ldots, X_d; Y_1,\ldots,Y_k}$ is a Banach multi-ideal in the sense of Pietsch  \cite{pietsch_ideals_1984} (see also \cite{fg}). We recall the definition: a subclass $\cal M$ of the class of multilinear operators between Banach spaces endowed with a function $\|\cdot\|_{\cal M} \colon {\cal M} \longrightarrow \mathbb{R}$ is a Banach multi-ideal if, for all $d \in \mathbb{N}$ and Banach spaces $E_1,\ldots,E_d,F$, the component  $$\mathcal{M}(E_1,\ldots,E_d;F) := \mathcal{L}(E_1,\ldots,E_d;F)\cap \mathcal{M}$$ enjoys the following:\\
$\bullet$ $\mathcal{M}(E_1,\ldots,E_d;F)$ is a linear subspace of  $\mathcal{L}(E_1,\ldots,E_d;F)$ containing the operators of the type 
    \begin{align} \label{finitetype}
        \varphi_1\otimes\cdots\otimes\varphi_d\otimes b\colon E_1\times \cdots\times E_d &\longrightarrow F ~,~(x^1,\ldots,x^d) \mapsto \varphi_1(x^1) \cdots \varphi_d(x^d) \cdot b,
    \end{align}
where $\varphi_m\in E_m^*$, $m=1,\ldots, d$, and $b\in F$; and, moreover,
\begin{equation}\label{opid}\|I^d\colon \mathbb{K}^d\longrightarrow\mathbb{K}~,~ I^d(\lambda^1,\ldots,\lambda^d) = \lambda^1\cdots\lambda^d\|_{\mathcal{M}} = 1.
\end{equation}
$\bullet$ $(\mathcal{M}(E_1,\ldots,E_d;F), \|\cdot\|_{\cal M})$ is a Banach space.\\
$\bullet$ The multi-ideal property: If $A\in \mathcal{M}(E_1,\ldots,E_d;F)$, $u_r\in \mathcal{L}(G_r;E_r)$, $r=1,\ldots,d$, and $t\in \mathcal{L}(F;H)$, then $t\circ A\circ (u_1,\ldots,u_d)\in \mathcal{M}(G_1,\ldots,G_d;H)$ and
\begin{align*}
        \| t\circ A\circ (u_1,\ldots,u_d)\|_{\mathcal{M}} \leq \| t\|\cdot \| A\|_{\mathcal{M}}\cdot\| u_1\|\cdots\| u_d\|.
    \end{align*}

We omit the proof that  $\mathcal{L}_{X_1,\ldots,X_d;Y_1,\ldots,Y_k}^{B_{\mathcal{I}}}$ is a linear subspace. 
We start by proving the multi-ideal property.

\begin{proposition}
\label{prop_propriedade_ideal2}
    If the sequence classes $X_1,\ldots, X_d, Y_1,\ldots,Y_k$ are linearly stable, then $(\mathcal{L}_{X_1,\ldots,X_d;Y_1,\ldots,Y_k}^{B_{\mathcal{I}}},$ ~$ \| \cdot\|_{B_{\mathcal{I}}; X_1,\ldots, X_d; Y_1,\ldots,Y_k})$ enjoys the multi-ideal property.
%
\end{proposition}
\begin{proof} Let $A\in\mathcal{L}_{X_1,\ldots,X_d;Y_1,\ldots,Y_k}^{B_{\mathcal{I}}}(E_1,\ldots,E_d;F)$, $u_r\in\mathcal{L}(G_r;E_r)$, $r=1,\ldots,d$, and $v\in\mathcal{L}(F;H)$ be given.
     For sequences  $(x_j^r)_{j=1}^{\infty}\in X_r(G_r)$, $r\in \{1,\ldots,d\}$, the linear stability of $X_r$ gives $(u_r(x_j^r))_{j=1}^{\infty}\in X_r(E_r)$, and since $A$ is $(B_{\mathcal{I}};X_1,\ldots, X_d;$ $Y_1,\ldots,Y_k)$-summing, 
     \begin{align*}
         Y_1(\cdots Y_k(F)\cdots) \ni& \widehat{A}_{B_{\mathcal{I}}}\left( (u_1(x_j^1))_{j=1}^{\infty},\ldots, (u_d(x_j^d))_{j=1}^{\infty}\right) =\\
         &=  \left(\cdots\left(  A\left( \sum_{s=1}^k\sum_{r\in I_s} u_r(x_{j_{n_s}^r}^r) * e_r\right) \right)_{n_k=1}^{\infty}\cdots\right)_{n_1=1}^{\infty} \\
         &= \left(\cdots\left(  A \circ (u_1,\ldots,u_d)\left( \sum_{s=1}^k\sum_{r\in I_s}x_{j_{n_s}^r}^r * e_r\right) \right)_{n_k=1}^{\infty}\cdots\right)_{n_1=1}^{\infty}.
    \end{align*}
 This proves that $A \circ (u_1,\ldots,u_d) \in \mathcal{L}_{X_1,\ldots,X_d;Y_1,\ldots,Y_k}^{B_{\mathcal{I}}}(G_1,\ldots,G_d;F)$ and that
    \begin{align*}
        \left( A \circ (u_1,\ldots,u_d)\right)^{\widehat{}}_{B_{\mathcal{I}}} = \widehat{A}_{B_{\mathcal{I}}}\circ (\widehat{u_1},\ldots, \widehat{u_d}),
    \end{align*}
    where $\widehat{u}_r \colon X_r(E_r)\longrightarrow X(F)$ is the linear operator induced by $u_r$.

Using the linear stability of $Y_1,\ldots, Y_k$ 
we have that the linear operator $\widehat{v}\colon Y_k(F)\longrightarrow Y_k(H)$ induced by $v$ is linear, continuous and $\|\widehat{v} \| = \|v\|$. The same happens with $\widehat{v}^2: =\widehat{\widehat{v}}\colon Y_{k-1}(Y_k(F))\longrightarrow Y_{k-1}(Y_k(H))$ and with its successive induced operators, until we reach the induced operator $\widehat{v}^k\colon Y_1(\cdots Y_k(F)\cdots)\longrightarrow Y_1(\cdots Y_k(H)\cdots)$,
    \begin{equation*}\widehat{v}^k\left(\cdots\left( y_{j_1, \ldots, j_k}\right)_{j_k = 1}^\infty \cdots \right)_{j_1= 1}^\infty = \left(\cdots\left( v(y_{j_1, \ldots, j_k})\right)_{j_k = 1}^\infty \cdots \right)_{j_1= 1}^\infty,\end{equation*}
which is linear continuous and $\|\widehat{v}^k\| = \|\widehat{v}^{k-1}\| = \cdots = \|\widehat{v}^2\| = \|v\|$. 
So, for sequences $(x_j^r)_{j=1}^{\infty}\in X_r(E_r)$, $r=1,\ldots,d$, since  $A\in\mathcal{L}_{X_1,\ldots,X_d;Y_1,\ldots,Y_k}^{B_{\mathcal{I}}}(E_1,\ldots,E_d;F)$, we have
    \begin{align*}
         \widehat{A}_{B_{\mathcal{I}}} \left( (x_j^1)_{j=1}^{\infty},\ldots,(x_j^d)_{j=1}^{\infty} \right) \in Y_1(\cdots Y_k(F)\cdots),
    \end{align*}
therefore,
    \begin{align*}
        Y_1(\cdots Y_k(H)\cdots)&\ni \widehat{v}^k\circ \widehat{A}_{B_{\mathcal{I}}} \left( (x_j^1)_{j=1}^{\infty},\ldots,(x_j^d)_{j=1}^{\infty} \right) \\
        &=\left(\ldots\left(  v\circ A\left( \sum_{s=1}^k\sum_{r\in I_s}x_{j_{n_s}^r}^r * e_r\right) \right)_{n_k=1}^{\infty}\ldots\right)_{n_1=1}^{\infty}.
    \end{align*}
proving that $v\circ A\in \mathcal{L}_{X_1,\ldots,X_d;Y_1,\ldots,Y_k}^{B_{\mathcal{I}}}(E_1,\ldots,E_d;H)$ and that $(v\circ A)_{B_{\mathcal{I}}}^{\widehat{}} = \widehat{v}^k\circ \widehat{A}_{B_{\mathcal{I}}}$. It follows that  
    \begin{align*}
        v\circ A\circ (u_1,\ldots,u_d) \in \mathcal{L}_{X_1,\ldots,X_d;Y_1,\ldots,Y_k}^{B_{\mathcal{I}}}(G_1,\ldots,G_d;H)
    \end{align*}
and 
    \begin{align*}
        \|  v\circ & A\circ (u_1,\ldots,u_d)\|_{B_{\mathcal{I}}; X_1,\ldots, X_d; Y_1,\ldots,Y_k} = \| \left( v\circ A\circ (u_1,\ldots,u_d)\right)_{B_{\mathcal{I}}}^{\widehat{}} \| \\
        &= \| \widehat{v}^k \circ \widehat{A}_{B_{\mathcal{I}}} \circ (\widehat{u_1},\ldots,\widehat{u_d})  \| \leq \| \widehat{v}^k\| \cdot\|  \widehat{A}_{B_{\mathcal{I}}} \| \cdot\prod_{r=1}^{d}\| \widehat{u}_r \| = \| v\| \cdot\|  A \|_{B_{\mathcal{I}}; X_1,\ldots, X_d; Y_1,\ldots,Y_k} \cdot\prod_{r=1}^{d}\| u_r \|.
    \end{align*}
\end{proof}

Next we identify a  H\"older-type condition which is necessary for $\mathcal{L}_{X_1,\ldots,X_d;Y_1,\ldots,Y_k}^{B_{\mathcal{I}}} \neq \{0\}$ and sufficient for $\mathcal{L}_{X_1,\ldots,X_d;Y_1,\ldots,Y_k}^{B_{\mathcal{I}}}$ to contain the operators of the type (\ref{finitetype}).

\begin{proposition}
\label{prop_ideal_nao_trivial_implica_I^d_no_ideal}
     Suppose that the sequence classes $X_1,\ldots,X_d,Y_1,\ldots,Y_k$ are linearly stable.\\
{\rm (a)} If there exists a nonzero operator $A\in\mathcal{L}_{X_1,\ldots, X_d; Y_1,\ldots,Y_k}^{B_{\mathcal{I}}}(E_1,\ldots , E_d;F)$, then
    \begin{align}
        \left(\cdots \left( \prod_{s=1}^k\prod_{r\in I_s} \lambda_{j_{n_s}^r}^r\right)_{n_k=1}^{\infty}\cdots \right)_{n_1=1}^{\infty}\in Y_1(\cdots Y_k(\mathbb{K})\cdots)\label{caracterizacao_operador_I^d}
    \end{align}
    whenever $(\lambda_j^r)_{j=1}^{\infty}\in X_r(\mathbb{K})$, $r=1,\ldots,d$.\\
{\rm (b)} If {\rm (\ref{caracterizacao_operador_I^d})} holds, then $\mathcal{L}_{X_1,\ldots, X_d; Y_1,\ldots,Y_k}^{B_{\mathcal{I}}}(E_1,\ldots , E_d;F)$ contains the operators of the type {\rm (\ref{finitetype})}.

\end{proposition}
\begin{proof}
    (a) Let $(x^1,\ldots,x^d)\in E_1\times \cdots\times E_d$ be such that $A(x^1,\ldots,x^d)\neq 0$. By the Hahn-Banach Theorem there exists  $\psi \in F^*$ such that $\psi(A(x^1,\ldots,x^d)) = \| A(x^1,\ldots,x^d)\|$. Taking successive induced operators as in the proof of 
    Proposition \ref{prop_propriedade_ideal2}(2), we end up with a bounded linear operator 
    $\widehat{\psi}^k\colon Y_1(\cdots Y_k(F)\cdots) \longrightarrow Y_1(\cdots Y_k(\mathbb{K})\cdots)$ given by
    $$\widehat{\psi}^k\left(\cdots\left( y_{j_1, \ldots, j_k}\right)_{j_k = 1}^\infty \cdots \right)_{j_1= 1}^\infty = \left(\cdots\left( \psi(y_{j_1, \ldots, j_k})\right)_{j_k = 1}^\infty \cdots \right)_{j_1= 1}^\infty.$$
For $r=1,\ldots,d$, let $u_r\in\mathcal{L}(\mathbb{K}; E_r)$ be defined by $u_r(\lambda) = \lambda x^r$, and consider the induced bounded linear operator 
$\widehat{u}_r\in \mathcal{L}(X(\mathbb{K}), X(E_r))$. 
By Proposition \ref{prop_propriedade_ideal2} it follows that $\psi \circ A\circ (u_1,\ldots,u_d)\in \mathcal{L}_{X_1,\ldots,X_d;Y_1,\ldots,Y_k}^{B_{\mathcal{I}}}(^d\mathbb{K};\mathbb{K})$. So, for $(\lambda_j^r)_{j=1}^{\infty}\in X_r(\mathbb{K})$, $r = 1, \ldots, d$,  \begin{align*}
Y_1(\cdots Y_k(\mathbb{K})\cdots) &\ni\left(\psi \circ A\circ (u_1,\ldots,u_d)\right)^{\widehat{}}_{B_{\mathcal{I}}} \left((\lambda_j^1)_{j=1}^{\infty},\ldots,(\lambda_j^d)_{j=1}^{\infty}\right) \\
         &=\| A(x^1,\ldots,x^d) \| \cdot \left(\ldots \left( \prod_{s=1}^k\prod_{r\in I_s} \lambda_{j_{n_s}^r}^r  \right)_{n_k=1}^{\infty} \ldots \right)_{n_1=1}^{\infty},
    \end{align*}
    where the equality results from a boring computation. Now use that $Y_1(\cdots Y_k(\mathbb{K})\cdots)$ is a linear space,

\noindent (b) Let $I^d\colon \mathbb{K}^d\longrightarrow\mathbb{K}$ be the operator of condition (\ref{opid}). Given scalar sequences $(\lambda_j^r)_{j=1}^{\infty}\in X_r(\mathbb{K})$, $r=1,\ldots,d$, since
\begin{align*}
        I^d\left( \sum_{s=1}^k\sum_{r\in I_s} \lambda_{j_{n_s}^r}^r*e_r\right)  = I^d\left( \sum_{r\in I_1} \lambda_{j_{n_1}^r}^r*e_r+\cdots+\sum_{r\in I_k} \lambda_{j_{n_k}^r}^r*e_r\right)
        = \prod_{s=1}^k\prod_{r\in I_s} \lambda_{j_{n_s}^r}^r
    \end{align*}
for all $n_1,\ldots,n_k\in\mathbb{N}$, the assumption (\ref{caracterizacao_operador_I^d}) gives that
    \begin{equation}\label{Desigid}\left(\cdots \left( I^d\left( \sum_{s=1}^k\sum_{r\in I_s} \lambda_{j_{n_s}^r}^r*e_r\right) \right)_{n_k=1}^{\infty}\cdots \right)_{n_1=1}^{\infty}  = \left(\cdots \left( \prod_{s=1}^k\prod_{r\in I_s} \lambda_{j_{n_s}^r}^r\right)_{n_k=1}^{\infty}\cdots \right)_{n_1=1}^{\infty}
    \end{equation}
belongs to $Y_1(\cdots Y_k(\mathbb{K})\cdots)$, proving that $I^d$ is $(B_{\mathcal{I}};X_1,\ldots,X_d;Y_1,\ldots,Y_k)$-summing.

For $\varphi_r\in E_r^*$, $r=1,\ldots,d$, and $b\in F$, considering the operator $u \colon\mathbb{K}\longrightarrow F$ given by $u(\lambda) = \lambda b$, by Proposition \ref{prop_propriedade_ideal2} we get  
      $$\varphi_1\otimes\cdots\otimes\varphi_d\otimes b = u\circ I^d\circ (\varphi_1,\ldots,\varphi_d) \in \mathcal{L}_{X_1,\ldots,X_d;Y_1,\ldots,Y_k}^{B_{\mathcal{I}}}(E_1,\ldots,E_d;F).$$
\end{proof}

Now it is quite natural to suppose the H\"older-type condition 
(\ref{caracterizacao_operador_I^d}) together with the corresponding norm inequality. 


\begin{definition}\rm
     Let $X_1,\ldots,X_d,Y_1,\ldots,Y_k$ be sequence classes. We say that the $(d+k)$-tuple $(X_1,\ldots,X_d,Y_1,\ldots,Y_k)$ is {\it $B_{\mathcal{I}}$-compatível} if, regardless of the scalar sequences  $(\lambda_j^r)_{j=1}^{\infty}\in X_r(\mathbb{K})$, $r=1,\ldots,d$, it holds
    \begin{align}
        \left(\cdots\left(\prod_{s=1}^k\prod_{r\in I_s} \lambda_{j_{n_s}^r}^r \right)_{n_k=1}^{\infty}\cdots\right)_{n_1=1}^{\infty} \in Y_1(\cdots(Y_k(\mathbb{K})\cdots) \mbox{~~and}
        \label{condicao_compatibilidade1}
    \end{align}
    \begin{align}
        \Big \|\left(\cdots\left(\prod_{s=1}^k\prod_{r\in I_s} \lambda_{j_{n_s}^r}^r \right)_{n_k=1}^{\infty}\cdots\right)_{n_1=1}^{\infty}\Big \|_{Y_1(\cdots(Y_k(\mathbb{K})\cdots)}\leq \prod_{r=1}^d\| (\lambda_j^r)_{j=1}^{\infty}\|_{X_r(\mathbb{K})}.
        \label{condicao_compatibilidade2}
    \end{align}

\end{definition}

In some cases the compatibility condition sounds familiar:

\begin{example}\rm
     For $q,p_1, \ldots, p_d \geq 1$, H\"older's inequality asserts that  $\frac{1}{q} \leq \frac{1}{p_1}+\cdots+\frac{1}{p_d}$  if and only if the  $(d+1)$-tuple $(\ell_{p_1}(\cdot), \ldots, \ell_{p_d}(\cdot); \ell_q(\cdot))$ is $\mathbb{N}^d$-compatible.
\end{example}

For further examples, see Propositions \ref{cap3-prop20}, \ref{cap3-prop21} and \ref{cap3-prop22}.

In the next lemma we need to regard a vector $(x^1, \ldots, x^d)\in E_1\times \ldots\times E_d$  by means of the partition $\mathcal{I} = \{ I_1,\ldots,I_k\}$ in the following way: 
    \begin{align}
        \sum_{s=1}^k\sum\limits_{r\in I_s}x^r *e_r &= \sum\limits_{r\in I_1}x^r *e_r + \cdots+ \sum\limits_{r\in I_k}x^r *e_r = (x^1,\ldots,x^d).\label{uytr}
    \end{align}

\begin{lemma}
\label{proposicao_desigualdade_normas}
     If the sequence classes $X_1,\ldots,X_d,Y_1,\ldots,Y_k$ are linearly stable and $A\in\mathcal{L}_{X_1,\ldots,X_d;Y_1,\ldots,Y_k}^{B_{\mathcal{I}}}(E_1,\ldots,E_d;F)$, then  $
        \| A\| \leq \| A\|_{B_{\mathcal{I}}; X_1,\ldots, X_d; Y_1,\ldots,Y_k}.$
\end{lemma}
\begin{proof}
     Given $(x^1,\ldots,x^d)\in E_1\times \cdots \times E_d$ and $r=1,\ldots,d$, consider the sequence $(y_j^r)_{j=1}^{\infty} =  x^r\cdot e_{j_1^r}$.  
So, $\sum\limits_{s=1}^k\sum\limits_{r\in I_s}y_{j_{n_s}^r}^r *e_r$ has some null coordinate if $n_s\neq 1$ for some $s\in \{1,\ldots,k\}$. If $n_1=\cdots = n_k = 1$, then $\sum\limits_{s=1}^k\sum\limits_{r\in I_s}y_{j_{n_s}^r}^r *e_r =  \sum\limits_{s=1}^k\sum\limits_{r\in I_s} x^r *e_r$. Thus, applying \cite[Lemma 2.4]{botelho_transformation_2017} repeated times, we get
    \begin{align*}
        \left\| \widehat{A}_{B_{\mathcal{I}}} \left((y_j^1)_{j=1}^{\infty}, \ldots, (y_j^d)_{j=1}^{\infty}\right)\right\|&_{Y_1(\cdots Y_k(F)\cdots)} =\\& = \Big\| \left( \cdots \left( A\left( \sum_{s=1}^k\sum_{r\in I_s} y_{j_{n_s}^r}^r *e_r\right) \right)_{n_k=1}^{\infty}\cdots \right)_{n_1=1}^{\infty}\Big\|_{Y_1(\cdots Y_k(F)\cdots)} \\
        &= \Big\| \left( \cdots \left( \left(  A\left( \sum_{s=1}^k\sum_{r\in I_s} y_{j_1^r}^r *e_r\right)\right) \cdot e_1\right)\cdots \right)\cdot e_1\Big\|_{Y_1(\cdots Y_k(F)\cdots)} \\
        &=\Big\|   A\left( \sum_{s=1}^k\sum_{r\in I_s} x^r *e_r\right)\Big\|_{F} = \left\|   A\left( x^1,\ldots,x^d\right)\right\|_{F},
    \end{align*}
    where the last equality comes from (\ref{uytr}). To complete the proof, note that   \begin{align*}
        \left\| A\left( x^1,\ldots,x^d\right)\right\|_{F} &= \left\| \widehat{A}_{B_{\mathcal{I}}} \left((y_j^1)_{j=1}^{\infty}, \ldots, (y_j^d)_{j=1}^{\infty}\right)\right\|_{Y_1(\cdots Y_k(F)\cdots)} \leq \| \widehat{A}_{B_{\mathcal{I}}}\| \cdot \prod_{r=1}^d\| (y_j^r)_{j=1}^{\infty}\|_{X_r(E_r)}\\
        &= \| \widehat{A}_{B_{\mathcal{I}}}\| \cdot \prod_{r=1}^d\| x^r\cdot e_{j_1^r}\|_{X_r(E_r)}=
         \| A \|_{B_{\mathcal{I}}; X_1,\ldots, X_d;Y_1,\ldots,Y_k} \cdot \prod_{r=1}^d\| x^r\|_{E_r}.
    \end{align*}
\end{proof}


\begin{proposition}
\label{prop_norma}
      If the sequence classes $X_1,\ldots,X_d,Y_1,\ldots,Y_k$ are linearly stable, then  
     the expression \eqref{eq_def_norma} defines a norm on the class $\mathcal{L}_{X_1,\ldots,X_d;Y_1,\ldots,Y_k}^{B_{\mathcal{I}}}$.
\end{proposition}
\begin{proof}  
    If $\| A \|_{B_{\mathcal{I}};X_1,\ldots,X_d;Y_1,\ldots,Y_k} =0$, then 
    $A = 0$ by Lemma \ref{proposicao_desigualdade_normas}. The remaining norm axioms follow from the linearity of the correspondence 
    $A \mapsto \widehat{A}_{B_{\mathcal{I}}}$.
\end{proof}


\begin{theorem}
     Let $X_1,\ldots,X_d,Y_1,\ldots,Y_k$ be $B_{\mathcal{I}}$-compatible linearly stable sequence classes. Then $(\mathcal{L}_{X_1,\ldots,X_d;Y_1,\ldots,Y_k}^{B_{\mathcal{I}}},\| \cdot\|_{B_{\mathcal{I}}; X_1,\ldots ,X_d;Y_1,\ldots, Y_k})$ is a Banach ideal of $d$-linear operators.
\end{theorem}
\begin{proof}
    Bearing in mind Propositions \ref{prop_ideal_nao_trivial_implica_I^d_no_ideal}(b), \ref{prop_propriedade_ideal2} and  \ref{prop_norma}, all that is left to prove is the equality 
    $\| I^d\|_{B_{\mathcal{I}}; X_1,\ldots ,X_d;Y_1,\ldots, Y_k} = 1$ and the completeness. For the equality, in $$1 = \|I^d\| \leq \| I^d\|_{B_{\mathcal{I}}; X_1,\ldots ,X_d;Y_1,\ldots, Y_k} \leq 1, $$
the first inequality comes from Lemma \ref{proposicao_desigualdade_normas} and the second from an obvious combination of (\ref{Desigid}) with (\ref{condicao_compatibilidade2}). 
The completeness of the space   $\mathcal{L}_{X_1,\ldots,X_d;Y_1,\ldots,Y_k}^{B_{\mathcal{I}}}(E_1,\ldots,E_d;F)$ with respect to the norm $\| \cdot\|_{B_{\mathcal{I}}; X_1,\ldots ,X_d;Y_1,\ldots, Y_k}$ can be proved by a standard argument using Lemma \ref{proposicao_desigualdade_normas} and the fact that convergence in a sequence class implies coordinatewise convergente, which follows from the assumption $X(\cdot) \stackrel{1}{\hookrightarrow} \ell_\infty(\cdot)$. We omit the (long and boring) details. 
\end{proof}

\section{Recovering the already studied classes}
\label{seccao_classes_estudadas}

In this section we show that the several classes of summing multilinear operators defined by the transformation of vector-valued sequences that have been studied in the literature are particular instances of our general construction. This includes, in particular, some classes that were not recovered in \cite{botelho_summing_2020}.



\subsection{$(X_1,\ldots,X_d;Y)$-summing operators}

Given sequence classes $X_1, \ldots, X_d, Y$, an operator  $A\in\mathcal{L}(E_1,\ldots,E_d;F)$ is $(X_1,\ldots,X_d;Y)$-summing according to \cite{botelho_transformation_2017} if $\left( A\left(x_n^1,\ldots,x_n^d\right)\right)_{n=1}^{\infty}  \in Y(F)$ whenever $(x_j^r)_{j=1}^{\infty}\in X_r(E_r)$, $r=1,\ldots,d$. The space $\mathcal{L}_{X_1,\ldots,X_d;Y}(E_1,\ldots,E_d;F)$ of such operators become a Banach space with the norm
$$\| A\|_{X_1,\ldots,X_d;Y} = \|\widehat{A} \colon X(E_1)\times\cdots\times X(E_d)\longrightarrow Y(F) \|. $$
Moreover, the class $\mathcal{L}_{X_1,\ldots,X_d;Y}$ is a Banach ideal of $d$-linear operators if the underlying sequence classes are linearly stable and satisfy the condition
 $X_1(\mathbb{K})\cdots X_d(\mathbb{K})\stackrel{1}{\hookrightarrow}Y(\mathbb{K})$, meaning that $(\lambda_j^1\cdots\lambda_j^d)_{j=1}^{\infty}\in Y(\mathbb{K})$ and $\| (\lambda_j^1\cdots\lambda_j^d)_{j=1}^{\infty}\|_{Y(\mathbb{K})}\leq \prod_{r=1}^d\| (\lambda_j^r)_{j=1}^{\infty}\|_{X_r(\mathbb{K})}$ whenever $(\lambda_j^r)_{j=1}^{\infty}\in X_r(\mathbb{K})$, $r=1,\ldots,d$.

To recover this class we recall the diagonal block $D(\mathbb{N}^d) = \{(j, \ldots, j) : j \in \mathbb{N}\}$ from Example \ref{exemplo_blocos_a} associated to the trivial partition and to the sequences $(j_n^1)_{n=1}^{\infty}=\cdots= (j_n^d)_{n=1}^{\infty}= (n)_{n=1}^{\infty}$.

\begin{proposition}\label{cap3-prop20}
    If the sequence classes $X_1,\ldots,X_d$ e $Y$ are linearly stable, then
$
        \mathcal{L}_{X_1,\ldots,X_d;Y}= \mathcal{L}_{X_1,\ldots,X_d;Y}^{D(\mathbb{N}^d)}$ isometrically and
 $X_1(\mathbb{K})\cdots X_d(\mathbb{K})\stackrel{1}{\hookrightarrow}Y(\mathbb{K})$ if and only if the $(d+1)$-tuple $(X_1,\ldots,X_d,Y)$ is $D(\mathbb{N}^d)$-compatible.
\end{proposition}
\begin{proof}   
We have $\mathcal{I}_t = \{I_1\}$ with $I_1 = \{1,\ldots,d\}$. For an operator $A \in {\cal L}(E_1, \ldots, E_d;F)$, $ A\in \mathcal{L}_{X_1,\ldots,X_d;Y}^{D(\mathbb{N}^d)}(E_1,\ldots,E_d;F)$ if and only if, for all sequences $(x_j^r)_{j=1}^{\infty}\in X_r(E_r)$, $r=1,\ldots,d$,
    \begin{align*}
        Y(F)\ni & \left( A\left(\sum_{s=1}^k \sum_{r\in I_s} x_{j_{n_s}^r}^r *e_r \right) \right)_{n_1=1}^{\infty} = \left( A\left(\sum_{r\in \{1,\ldots,d\}} x_{j_{n_1}^r}^r *e_r \right) \right)_{n_1=1}^{\infty} \\
        &= \left( A\left(\sum_{r\in \{1,\ldots,d\}} x_{n_1}^r *e_r \right) \right)_{n_1=1}^{\infty} = \left( A\left(x_n^1,\ldots,x_n^d\right) \right)_{n=1}^{\infty},
    \end{align*}
if and only if $A$ is $(X_1,\ldots,X_d;Y)$-summing. The equality of norms follows from $\widehat{A} = \widehat{A}_{D(\mathbb{N}^d)}$. For the second statement, note that
   $X_1(\mathbb{K})\cdots X_d(\mathbb{K})\stackrel{1}{\hookrightarrow}Y(\mathbb{K})$ if and only if $I^d\in \mathcal{L}_{X_1,\ldots,X_d;Y}(\mathbb{K}^d;\mathbb{K})$ and $\|I^d\|_{X_1, \ldots, X_d;Y} \leq 1$ if and only if $(X_1,\ldots,X_d,Y) $ is $D(\mathbb{N}^d)$-compatible.
    \end{proof}

We list a few classes of summing operators that Proposition \ref{cap3-prop20} recovers as particular cases of our construction: \\
\noindent $\bullet$ The class of absolutely $(q;p_1, \ldots, p_d)$-summing $d$-linear operators 
defined in the introduction. In our terminology, this is the class of $(D(\mathbb{N}^d);\ell_{p_1}^w(\cdot), \ldots, \ell_{p_d}^w(\cdot); \ell_q(\cdot))$-summing operators.\\
\noindent$\bullet$ The class of cotype $q$ multilinear operators, $q \geq 1$, considered in \cite{botelho_type_2016}, which is the class of 
$(D(\mathbb{N}^d);{\rm Rad}(\cdot), \ldots, {\rm Rad}(\cdot)); \ell_q(\cdot))$-summing operators in our notation.\\
\noindent$\bullet$ The class of weakly summing $(q; p_1, \ldots, p_d)$-summing $d$-linear operators, $\frac{1}{q} \leq \frac{1}{p_1} + \cdots + \frac{1}{p_d}$, studied in, e.g., \cite{blasco_coincidence_2016,botelho_transformation_2017,kim_multiple_2007,popa_mixing_2012}. This is the class of $(D(\mathbb{N}^d);\ell_{p_1}^w(\cdot), \ldots, \ell_{p_d}^w(\cdot); \ell_q^w(\cdot))$-summing operators.\\
\noindent$\bullet$ The class of type $(p_1, \ldots, p_d)$ $d$-linear operators, $\frac{1}{2} \leq \frac{1}{p_1} + \cdots +\frac{1}{p_d} < 1$, treated in \cite{botelho_type_2016}, which is the class of  $(D(\mathbb{N}^d);\ell_{p_1}(\cdot), \ldots, \ell_{p_d}(\cdot); {\rm Rad}(\cdot))$-summing operators.

The first two classes above are also encompassed by the construction of \cite{botelho_summing_2020}, the last two ones are not.

\subsection{Multiple $(\gamma_{s};X_1,\ldots,X_d)$-summing operators}

Here we show that all classes of multilinear operators studied in
\cite{ribeiro_generalized_2019} are particular instances of our geneeral construction. 
This reinforces that this constuction goes far beyond the one in \cite{botelho_summing_2020}, because the one and only class among the ones considered in \cite{ribeiro_generalized_2019} that is recovered in \cite{botelho_summing_2020} is the class of multiple $(q;p_1, \ldots, p_d)$-summing operators.

To recover the general case from \cite{ribeiro_generalized_2019} we need to recall the following notions: \\
$\bullet$ 
For $d\in\mathbb{N}$, a {\it $d$-sequence} in $E$ is a function $g \colon \mathbb{N}^d\longrightarrow E$. Writing 
        $g(j^1,\ldots,j^d) = x_{j^1.\ldots,j^d}$ for all 
        $j^1,\ldots,j^d\in\mathbb{N}$,
 $g$ can be represented by $(x_{j^1,\ldots,j^d})_{j^1,\ldots,j^d=1}^{\infty}$.\\
$\bullet$ 
$    c_{00}(E;\mathbb{N}^d)\hiderel{:=}\left\{ (x_{j^1,\ldots,j^d})_{j^1,\ldots,j^d=1}^{\infty}: x_{j^1,\ldots,j^d}\neq 0 \mbox{ only for finitely many } j^1,\ldots,j^d\right\}.$\\
$\bullet$ 
    $\ell_{\infty}(E;\mathbb{N}^d)=\left\{ (x_{j^1,\ldots,j^d})_{j^1,\ldots,j^d=1}^{\infty}: \sup_{j^1,\ldots,j^d\in\mathbb{N}}\Vert x_{j^1,\ldots,j^d}\Vert_E <\infty \right\}.$\\
$\bullet$ For $k_1, \ldots, k^d \in \mathbb{N}$, $e_{k^1,\ldots,k^d}\colon \mathbb{N}^d\longrightarrow \mathbb{K}$ is the $d$-sequence given by
\begin{align*}
    e_{k^1,\ldots,k^d}(j^1,\ldots,j^d) =\begin{cases} 1, & \mbox{ if } j^1=k^1,\ldots,j^d=k^d\\ 0, &\mbox{ otherwise. }\end{cases}
\end{align*}
$\bullet$ 
A {\it $d$-sequence class} is a map  $\gamma_s(\cdot;\mathbb{N}^d)$ that assigns to each Banach space  $E$ a Banach space $\gamma_s(E;\mathbb{N}^d)$ of $E$-valued $d$-sequences with the coordinatewise algebraic operations such that 
  $c_{00}(E;\mathbb{N}^d)\subset \gamma_s(E;\mathbb{N}^d)\stackrel{1}{\hookrightarrow}\ell_{\infty}(E;\mathbb{N}^d)$ and
        $\Vert e_{k^1,\ldots,k^d}\Vert_{\gamma_s(\mathbb{K};\mathbb{N}^d)}=1$ for all $k^1,\ldots,k^d\in \mathbb{N}$. \\
$\bullet$ Given a $d$-sequence class $\gamma_s(\cdot;\mathbb{N}^d)$ and sequence classes $X_1, \ldots, X_d$, according to \cite{ribeiro_generalized_2019} an operator $A \in {\cal L}(E_1, \ldots, E_d;F)$ is said to be {\it multiple $(\gamma_{s};X_1,\ldots,X_d)$-summing} if $
    \left(A\left(x_{j^1}^1,\ldots,x_{j^d}^d\right)\right)_{j^1,\ldots,j^d=1}^{\infty}\in \gamma_s(F;\mathbb{N}^d)$ whenever $\left(x_j^r\right)_{j=1}^{\infty}\in X_r(E_r)$, $r = 1,\ldots, d$. 
The set $\mathcal{L}_{\gamma_{s};X_1,\ldots,X_d}^m(E_1,\ldots,E_d;F)$ of all such operators becomes a Banach space 
with the norm $\|\cdot\|_{\mathcal{L}_{\gamma_{s};,X_1,\ldots,X_d}^m}$ given by the norm of the induced $d$-linear operator $\widehat{A} \colon  X_1(E_1) \times \cdots \times  X_d(E_d) \longrightarrow$ $\gamma_s(F;\mathbb{N}^d)$.


The key to recover these spaces from our framework is to establish an interplay between sequence classes and $d$-sequence classes. 
We do that considering the block $B = \mathbb{N}^d$ as in Example \ref{exemplo_blocos_c}, defined by a fixed bijection $\sigma \colon  \mathbb{N}\longrightarrow \mathbb{N}^d$ and sequences of natural numbers $(j_n^r)$, $r=1,\ldots,d$, where 
$\sigma(n) = (j_n^1,\ldots,j_n^d)$ for every $n\in\mathbb{N}$. We omit the proof of the following lemma.

\begin{lemma}\label{lemma_contrucao_classe_sequencia}
    Given a $d$-sequence class $\gamma_s(\cdot;\mathbb{N}^d)$, the correspondence  $E \mapsto Y_{\gamma_s}(E)$, where
    \begin{align}\label{eq_lemma_contrucao_classe_sequencia}
        Y_{\gamma_s}(E) :=\{ \left(g\circ \sigma (n)\right)_{n=1}^{\infty}\in E^{\mathbb{N}} :   g\in \gamma_s(E;\mathbb{N}^d) \} ~,~\| \left(g\circ \sigma (n)\right)_{n=1}^{\infty} \|_{Y_{\gamma_s}(E)} := \| g\|_{\gamma_s(E;\mathbb{N}^d)},
        \end{align}
is a sequence class.
\end{lemma}

To get a Banach ideal of multilinear operators, the following condition was imposed on the class
$\mathcal{L}_{\gamma_s;X_1,\ldots,X_d}^m$  in \cite{ribeiro_generalized_2019}: the symbol 
$    X_1(\mathbb{K})\cdots X_d(\mathbb{K})\stackrel{mult,1}{\hookrightarrow}\gamma_s(\mathbb{K};\mathbb{N}^d)$ means that
$\left(\lambda_{j_1}^1\cdots\lambda_{j_d}^d\right)_{j_1,\ldots,j_d=1}^{\infty}\in \gamma_s(\mathbb{K};\mathbb{N}^d)$ and $$\Big\| \left(\lambda_{j_1}^1\cdots\lambda_{j_d}^d\right)_{j_1,\ldots,j_d=1}^{\infty} \Big\|_{\gamma_s(\mathbb{K};\mathbb{N}^d)}\leq \prod_{r=1}^d\Big\| \left(\lambda_j^r\right)_{j=1}^{\infty} \Big\|_{X_r(E_r)}$$
whenever $\left(\lambda_j^r\right)_{j=1}^{\infty} \in X_r(\mathbb{K})$, $r = 1, \ldots, d$. The next recovers shows that our construction recovers the general case from 
\cite{ribeiro_generalized_2019}. 
Recall that $\mathcal{L}_{X_1,\ldots,X_d;Y_{\gamma_s}}^{\mathbb{N}^d}$ is the class of $(\mathbb{N}^d; X_1, \ldots, X_s; Y_{\gamma_s})$-summing operators in the sense of Section \ref{sec-2.2-construcao}.
\begin{proposition}
\label{cap3-prop21} Let $X_1,\ldots, X_d$ be sequence classes and $\gamma_s(\cdot;\mathbb{N}^d)$ be a $d$-sequence class such that  $X_1(\mathbb{K})\cdots X_d(\mathbb{K})\stackrel{mult,1}{\hookrightarrow}\gamma_s(\mathbb{K};\mathbb{N}^d)$. Then the $(d+1)$-tuple  $(X_1,\ldots,X_d,Y_{\gamma_s})$ is $\mathbb{N}^d$-compatible and     \begin{align*} \mathcal{L}_{X_1,\ldots,X_d;Y_{\gamma_s}}^{\mathbb{N}^d}(E_1,\ldots,E_d;F)= \mathcal{L}_{\gamma_s;X_1,\ldots,X_d}^m(E_1,\ldots,E_d;F)
\end{align*}
isometrically for all Banach spaces $E_1, \ldots, E_d,F$.
\end{proposition}

\begin{proof} Let $\left(\lambda_j^r\right)_{j=1}^{\infty} \in X_r(\mathbb{K})$, $r=1,\ldots,d$, be given. The assumption $X_1(\mathbb{K})\cdots X_d(\mathbb{K})\stackrel{mult,1}{\hookrightarrow}\gamma_s(\mathbb{K};\mathbb{N}^d)$ gives that  $\left(\lambda_{j_1}^1\cdots\lambda_{j_d}^d\right)_{j_1,\ldots,j_d=1}^{\infty}\in \gamma_s(\mathbb{K};\mathbb{N}^d)$. Defining $g\colon \mathbb{N}^d\longrightarrow\mathbb{K}$ by $g(j^1,\ldots,j^d) = \lambda_{j^1}^1\cdots\lambda_{j^d}^d,$ by  (\ref{eq_lemma_contrucao_classe_sequencia}) we get $\left(g \circ \sigma(n) \right)_{n=1}^{\infty} \in Y_{\gamma_s}(\mathbb{K}) $. Therefore,
    \begin{align*}
        \left(\lambda_{j_n^1}^1\cdots\lambda_{j_n^d}^d\right)_{n=1}^{\infty}  = \left(g(j_n^1, \ldots, j_n^d )\right)_{n=1}^{\infty} = \left(g \circ \sigma(n) \right)_{n=1}^{\infty} \in Y_{\gamma_s}(\mathbb{K}) \mbox{~~e}
    \end{align*}
    \begin{align*}
        \Big\Vert \left(\lambda_{j_n^1}^1\cdots\lambda_{j_n^d}^d\right)_{n=1}^{\infty} \Big\Vert_{Y_{\gamma_s}(\mathbb{K})} &= \Vert \left(g \circ \sigma(n) \right)_{n=1}^{\infty}\Vert_{Y_{\gamma_s}(\mathbb{K})} = \Vert g\Vert_{\gamma_s(\mathbb{K};\mathbb{N}^d)} \\
        &=  \Big\| \left(\lambda_{j_1}^1\cdots\lambda_{j_d}^d\right)_{j_1,\ldots,j_d=1}^{\infty} \Big\|_{\gamma_s(\mathbb{K};\mathbb{N}^d)}\leq \prod_{r=1}^d\Big\| \left(\lambda_j^r\right)_{j=1}^{\infty} \Big\|_{X_r(E_r)},
    \end{align*}
proving that $(X_1,\ldots,X_d,Y_{\gamma_s})$ is $\mathbb{N}^d$-compatible.

    Let the sequences $(x_j^r)_{j=1}^{\infty}\in X_r(E_r)$, $r=1,\ldots,d$, be given. On the one hand, if $A \in \mathcal{L}_{X_1,\ldots,X_d;Y_{\gamma_s}}^{\mathbb{N}^d}(E_1,\ldots,E_d;F)$ then $\left( A\left( x_{j_n^1}^1,\ldots, x_{j_n^d}^d\right)\right)_{n=1}^{\infty}\in Y_{\gamma_s}(F)$, hence there exists $g\in \gamma_s(F;\mathbb{N}^d)$ such that $(g\circ \sigma(n))_{n=1}^{\infty} = \left( A\left( x_{j_n^1}^1,\ldots, x_{j_n^d}^d\right)\right)_{n=1}^{\infty}$. 
           For every $(j^1,\ldots,j^d)\in\mathbb{N}^d$ there is a unique natural number $n$ such that $(j^1,\ldots,j^d) = \sigma(n)$, so $$A\left( x_{j^1}^1,\ldots, x_{j^d}^d\right)= g\circ \sigma(n)= g(j^1,\ldots,j^d).$$ It follows that 
    \begin{align*}
        \gamma_s(F;\mathbb{N}^d)\ni (g(j^1,\ldots,j^d))_{j^1,\ldots,j^d=1}^{\infty} =  (A\left( x_{j^1}^1,\ldots, x_{j^d}^d\right))_{j^1,\ldots,j^d=1}^{\infty},
    \end{align*}
 proving that $A\in \mathcal{L}_{\gamma_{s};X_1,\ldots,X_d}^m(E_1,\ldots,E_d;F)$. On the other hand, suppose that $A\in \mathcal{L}_{\gamma_{s};X_1,\ldots,X_d}^m(E_1,\ldots,E_d;F)$. Defining
      $$g \colon \mathbb{N}^d \longrightarrow F~,~g(j^1,\ldots,j^d) = A\left( x_{j^1}^1,\ldots, x_{j^d}^d\right),$$
we have $g \in \gamma_s(F;\mathbb{N}^d),$ from which it follows that 
    \begin{align*}
        Y(F)\ni (g\circ \sigma(n))_{n=1}^{\infty} = \left( g(j_n^1,\ldots,j_n^d)\right)_{n=1}^{\infty} =  \left( A\left( x_{j_n^1}^1,\ldots, x_{j_n^d}^d\right)\right)_{n=1}^{\infty},
    \end{align*}
showing that $A \in \mathcal{L}_{X_1,\ldots,X_d;Y_{\gamma_s}}^{\mathbb{N}^d}(E_1,\ldots,E_d;F)$. Moreover,
    \begin{align*}
        \| A\|_{\mathcal{L}_{\gamma_{s};X_1,\ldots,X_d}^m} &= \| \widehat{A}: X_1(E_1)\times\cdots\times X_d(E_d)\longrightarrow \gamma_s(F;\mathbb{N}^d) \| \\
        &= \sup\{\| \left(A\left(x_{j^1}^1,\ldots,x_{j^d}^d \right)\right)_{j^1,\ldots,j^d=1}^{\infty}\|_{\gamma_s(F;\mathbb{N}^d)}: (x_j^r)_{j=1}^{\infty}\in B_{X_r(E_r)}, r=1,\ldots,d\}\\
        &= \sup\{\| g\|_{\gamma_s(F;\mathbb{N}^d)}: (x_j^r)_{j=1}^{\infty}\in B_{X_r(E_r)}, r=1,\ldots,d\}\\
        &= \sup\{\| (g\circ \sigma(n))_{n=1}^{\infty}\|_{Y_{\gamma_s}(F)}: (x_j^r)_{j=1}^{\infty}\in X_r(E_r), r=1,\ldots,d\}\\
        &= \sup\Big\{\Big\|  \left( A\left( x_{j_n^1}^1,\ldots, x_{j_n^d}^d\right)\right)_{n=1}^{\infty}\Big\|_{Y_{\gamma_s}(F)}: (x_j^r)_{j=1}^{\infty}\in B_{X_r(E_r)}, r=1,\ldots,d\Big\}\\
        &= \| A\|_{\mathbb{N}^d;X_1,\ldots,X_d;Y_{\gamma_s}}.
    \end{align*}
\end{proof}

\begin{example}\label{Exex}\rm Among the classes encompassed by \cite{ribeiro_generalized_2019}, which are recovered by our construction by Proposition \ref{cap3-prop21}, we mention the following:

\noindent(a) The class of multiple $(p,q_1,\ldots,q_d)$-summing multilinear operators studied in, e.g., \cite{bayart_multiple_2018, botelho_inclusions_2009, botelho_when_2009, defant_coordinatewise_2010, montanaro_applications_2012, perez-garcia_unbounded_2008}, which, in our terminology, is the class of $(\mathbb{N}^d; \ell_{q_1}^w(\cdot),\ldots,\ell_{q_d}^w(\cdot);\ell_p(\cdot))$-summing operators.\\
(b) The class of strongly Cohen multiple $p$-summing operators introduced in  
 \cite{campos_cohen_2014}, which is the class of 
$(\mathbb{N}^d;\ell_p(\cdot), \stackrel{(d)}{\cdots}, \ell_p(\cdot);\ell_p\langle\cdot\rangle)$-summing operators.\\
(c) The class of weakly multiple mid $(p;q_1,\ldots,q_d)$-summing operators introduced in \cite{botelho_operator_2017}.
Investigating the sequence spaces introduced in \cite{karnsinha}, the sequence class $\ell_p^{\rm mid}(\cdot)$, $1\leq p<\infty$, of mid $p$-summable sequences was defined in 
\cite{botelho_operator_2017}. The class of multilinear operators we are referring to is the class of $(\mathbb{N}^d;\ell_{q_1}^w(\cdot),\ldots,\ell_{q_d}^w(\cdot);\ell_p^{\rm mid}(\cdot))$-summing operators.

The following two examples were first studied in \cite{ribeiro_generalized_2019}.

\noindent(d) The class of strongly multiple $(s,q,p)$-mixing operators, $1\leq q\leq s\leq \infty$, $p\leq q$. This is the class of 
$(\mathbb{N}^d;\ell_{m(s,q)}(\cdot),\stackrel{d}{\ldots},\ell_{m(s,q)}(\cdot); \ell_p(\cdot))$-summing operators, where $\ell_{m(s,q)}(E)$ is the space of mixed $(s,q)$-summable $E$-valued sequences (see \cite{pietsch_operator_1978}).\\
(e) The class of strongly multiple mid $p$-summing operators, $1\leq p<\infty$, which is the class of 
$(\mathbb{N}^d;\ell_p^{\rm mid}(\cdot),\stackrel{d}{\ldots}, \ell_p^{\rm mid}(\cdot); \ell_p(\cdot))$-summing operators.
\end{example}


\subsection{Partially multiple summing operators}\label{secao_parcialmente_multiplo_somante}
This class of operators was investigated in 
\cite{albuquerque_summability_2018, albuquerque_summability_2019, araujo_classical_2016} and it was the main motivation for Definition 
\ref{definicao-operadores somantes}. Only a few specific cases of this class were recovered in \cite{botelho_summing_2020}. Here we will recover the general case of this class as a particular instance of our constuction. 


    Fixed a partition  $\mathcal{I} = \{I_1, \ldots, I_k\}$ of $\{1,\ldots,d\}$ and given $(\textbf{p},\textbf{q}):=( p_1,\ldots,p_d,q_1,\ldots,q_k)\in [1,\infty)^{d+k}$, $\frac{1}{q_s}\leq \sum\limits_{r\in I_s}\frac{1}{p_r}$ for every $s\in\{1,\ldots , k\}$, a $d$-linear operator $A\colon E_1 \times \cdots \times E_d \longrightarrow F$ is said to be $\mathcal{I}$-partially multiple $(\textbf{q};\textbf{p})$-summing if there exists a constant $C>0$ such that
    \begin{align*}
        \left(\sum_{n_1=1}^{\infty}\left(\cdots\left(\sum_{n_k=1}^{\infty}\Big\| A\left(\sum_{s=1}^k\sum_{r\in I_s} x_{n_s}^r *e_r\right)\Big\|^{q_k}\right)^{\frac{q_{k-1}}{q_k}}\cdots \right)^{\frac{q_1}{q_2}}\right)^{\frac{1}{q_1}} \leq C \cdot \prod_{r=1}^d\| (x_j^r)_{j=1}^{\infty}\|_{w,p_r}
    \end{align*}
for any sequences $(x_j^r)_{j=1}^{\infty}\in \ell_{p_r}^w(E_r)$, $r=1,\ldots,d$. The $\mathcal{I}$-partially multiple $(\textbf{q};\textbf{p})$-summing norm of $A$ is defined as the infimum of all such constants $C$ and denoted by  $\pi_{(\textbf{q};\textbf{p})}(A)$.
The Banach space of all such operators is denoted by 
$\prod_{(\textbf{q};\textbf{p})}^{k,d,\mathcal{I}}(E_1,\ldots,E_d;F)$.

\begin{proposition}\label{cap3-prop22}
    For all Banach spaces $E_1,\ldots,E_d,F$,
    \begin{align*}
        \prod_{(\textbf{q};\textbf{p})}^{k,d,\mathcal{I}}(E_1,\ldots,E_d;F) = \mathcal{L}_{\ell_{p_1}^w(\cdot),\ldots,\ell_{p_d}^w(\cdot);\ell_{q_1}(\cdot),\ldots,\ell_{q_k}(\cdot)}^{B_{\mathcal{I}}}(E_1,\ldots,E_d;F)
    \end{align*}
isometrically, where $B_{\mathcal{I}}$ is the block associated to the partition $\mathcal{I}$ and to the sequences $(j_n^1)_{j=1}^{\infty} = \cdots = (j_n^d)_{j=1}^{\infty} = (n)_{n=1}^{\infty}$. Moreover,   $\frac{1}{q_s}\leq \sum\limits_{r\in I_s}\frac{1}{p_r}$ for every $s\in\{1,\ldots , k\}$ if and only if the  $(d+k)$-tuple  $(\ell_{p_1}^w(\cdot),\ldots,\ell_{p_d}^w(\cdot);\ell_{q_1}(\cdot),\ldots,\ell_{q_k}(\cdot))$ is $B_{\mathcal{I}}$-compatible.
\end{proposition}
\begin{proof}
    For an operator $A \in {\cal L}(E_1, \ldots, E_d;F)$, $A\in \mathcal{L}_{\ell_{p_1}^w(\cdot),\ldots,\ell_{p_d}^w(\cdot);\ell_{q_1}(\cdot),\ldots,\ell_{q_k}(\cdot)}^{B_{\mathcal{I}}}(E_1,\ldots,E_d;F)$ if and only if
    \begin{align*}
        \ell_{q_1}(\ldots\ell_{q_k}(F)\ldots)&\ni \left( \cdots \left( A\left( \sum_{s=1}^k \sum_{r\in I_s} x_{j_{n_s}^r}^r *e_r \right) \right)_{n_k=1}^{\infty} \ldots\right)_{n_1=1}^{\infty} = \\
        &= \left( \cdots \left( A\left( \sum_{s=1}^k \sum_{r\in I_s} x_{n_s}^r *e_r \right) \right)_{n_k=1}^{\infty} \ldots\right)_{n_1=1}^{\infty}
    \end{align*}
 whenever $(x_j^r)_{j=1}^{\infty}\in \ell_{p_r}^w(E_r)$, $r=1,\ldots,d$, if and only if 
 $A\in \prod_{(\textbf{q};\textbf{p})}^{k,d,\mathcal{I}}(E_1,\ldots,E_d;F)$. Since the induced operators are the same, the norms coincide. 
Given scalar sequences $(\lambda_j^r)_{j=1}^{\infty}\in \ell_{p_r}^w(\mathbb{K}) =  \ell_{p_r}$, $r=1,\ldots,d$, and writing $I_s = \{ r_1^s,\ldots,r_{m_s}^s\}$ for every $s=1,\ldots,k$, it is not difficult to check that
    \begin{align*}
       \Big\| \left(\cdots \left( \prod_{s=1}^k\prod_{r\in I_s} \lambda_{n_s}^r\right)_{n_k=1}^{\infty}\cdots\right)_{n_1=1}^{\infty}&\Big\|_{ \ell_{q_1}(\cdots\ell_{q_k}(\mathbb{K})\cdots)} \\&=
        \Big\| \left(\prod_{i=1}^{m_1} \lambda_{n_1}^{r_i^1} \right)_{n_1=1}^{\infty}\Big\|_{\ell_{q_1}}\cdots\Big\| \left(\prod_{i=1}^{m_k} \lambda_{n_k}^{r_i^k} \right)_{n_k=1}^{\infty}\Big\|_{\ell_{q_1}}.
    \end{align*}
Now the second statement follows from H\"older's inequality and the usual inequalities between $\ell_p$-norms.
\end{proof}

In the next example we recover a class that was studied in
\cite{albuquerque_summability_2018,albuquerque_summability_2019, albuquerque_sharp_2014, araujo_classical_2016} as an anisotropic generalization of the multiple $(q;p_1,\ldots,p_d)$-summing operators of Example \ref{Exex}(a)). 

\begin{example}\rm For $1\leq p_1,\ldots,p_d,q_1,\ldots,q_d<\infty$ with $p_r\leq q_r$, $r=1,\ldots,d$, an operator $A\in\mathcal{L}(E_1,\ldots,E_d;F)$ is said to be multiple $(q_1,\ldots,q_d;p_1,\ldots,p_d)$-summing if
    \begin{align*}
        \sum_{n_1=1}^{\infty}\left(\sum_{n_2=1}^{\infty}\cdots \left(\sum_{n_d=1}^{\infty} \| A(x_{n_1}^1,\ldots,x_{n_d}^d)\|_F^{q_d}\right)^{\frac{q_{d-1}}{q_d}}\cdots \right)^{\frac{q_1}{q_2}}<\infty,
    \end{align*}
whenever $(x_j^r)_{j=1}^{\infty}\in \ell_{p_r}^w(E_r)$, $r=1,\ldots,d$. This is the class recovered in Proposition \ref{cap3-prop22} for the discrete partition ${\cal I}_d$. 

\end{example}

\subsection{$\Lambda$-$(r,p)$-summing operators} \label{bayartepopa}

In this subsection we show that the class of operators recently studied, independently, in
\cite{bayart_coincidence_2020} and  \cite{popa_operators_2020} is also a particular case of our construction. 

     Let $p,q\geq 1$ and $\Lambda\subset\mathbb{N}^d$ be given. An operator $A\in \mathcal{L}(E_1,\ldots, E_d;F)$ is said to be $\Lambda$-$(q,p)$-summing if there exists a constant $C>0$ such that, for all sequences, $(x_j^r)_{j=1}^{\infty}\in \ell_p^w(E_r)$, $r=1,\ldots,d$, it holds
    \begin{align}
        \left( \sum_{(j^1,\ldots,j^d)\in\Lambda}\| A(x_{j^1}^1,\ldots,x_{j^d}^d)\|^q \right)^{\frac{1}{q}}\leq C\prod_{r=1}^d\| (x_j^r)_{j=1}^{\infty}\|_{\ell_p^w(E_r)}.
        \label{des-2.12}
    \end{align}

The $\Lambda$-$(q,p)$-summing norm of $A$ is defined as the infimum of these constants $C$ and denoted by
$\pi_{q,p}^{\Lambda}(\cdot)$. Let us see that this class coincides isometrically with our class $\mathcal{L}_{\ell_p^w(\cdot),\stackrel{d}{\ldots},\ell_p^w(\cdot);\ell_q(\cdot)}^{\Lambda}$ if $\Lambda$ is infinite.

According to Example \ref{exemplo_blocos_c}, let us fix a bijection $\sigma \colon \mathbb{N}\longrightarrow\Lambda$ and consider, for $r=1,\ldots,d$, the sequence $(j_n^r)_{n=1}^{\infty}$ in which $j_n^r$ is the $r$-th coordinate of  the $d$-tuple $\sigma(n)\in\Lambda$. Henceforth we consider $\Lambda$ as the block  associated to the trivial partition $\mathcal{I}_t$ and to the aforementioned sequences of natural numbers. 
Note that the construction of $\mathcal{L}_{\ell_p^w(\cdot),\stackrel{d}{\ldots},\ell_p^w(\cdot);\ell_q(\cdot)}^{\Lambda}$ does not depend on the bijection $\sigma$ because the sequence class 
$\ell_q(\cdot)$ is symmetric (cf. Remark \ref{obs-2}).

An operator $A\in {\cal L}(E_1,\ldots,E_d;F)$ is $\Lambda$-$(r,p)$-summing $\Longleftrightarrow$ there is $C > 0 $  as in (\ref{des-2.12}) $\Longleftrightarrow$ there is $C > 0$ such that 
    \begin{align}
        \left( \sum_{n=1}^{\infty}\| A(x_{j_n^1}^1,\ldots,x_{j_n^d}^d)\|^q \right)^{\frac{1}{q}}\leq C\prod_{r=1}^d\| (x_j^r)_{j=1}^{\infty}\|_{\ell_p^w(E_r)}
        \label{eq_lambia_daniel_2}
    \end{align}
whenever $(x_j^r)_{j=1}^{\infty} \in \ell_p^w(E_r)$, $r=1,\ldots d$  (we are using that the series $\sum\limits_{(j_1,\ldots,j_d)\in\Lambda}\| A(x_{j_1}^1,\ldots,x_{j_d}^d)\|^q$ is unconditionally convergent) $\Longleftrightarrow$ $
A\in \mathcal{L}_{\ell_p^w(\cdot),\stackrel{d}{\ldots},\ell_p^w(\cdot);\ell_q(\cdot)}^{\Lambda}(E_1,\ldots,E_d;F).$ Furthermore,
    \begin{align*}
        \pi_{q,p}^{\Lambda}(A) &= \inf\{ C>0: \mbox{ (\ref{des-2.12}) holds}\}= \inf\{ C>0: \mbox{ (\ref{eq_lambia_daniel_2}) holds}\} \\& =\| \widehat{A}_{\Lambda}\| = \| A\|_{\Lambda; \ell_p^w(\cdot),\stackrel{d}{\ldots},\ell_p^w(\cdot);\ell_q(\cdot)}.
    \end{align*}

\section{Downward coherence and coincidence results}

Coherence of multi-ideals and coincidence results are classical topics in the theory of multilinear operators 
(see  \cite{carando_coherent_2009,carando_every_2012, carando_holomorphic_2012, pellegrino_almost_2012,pellegrino_multi-ideals_2014,ribeiro_absolutely_2021,ribeiro_coherence_2020}). They make the interplay between multi-ideals/polynomials ideals and holomorphy types (see \cite{{botelho_holomorphy_2006}, campos_cohen_2014, carando_coherent_2009, carando_holomorphic_2012}), and applications to linear dynamics have been found 
(recent applications can be found in \cite{blas20201, blas20202, muro2017, muro2018}). In this section we provide applications of our construction to these issues. 


The symbol $[\stackrel{r}{\ldots}]$ means that the $r$-th coordinate has been omitted, for example: 
 $$E_1\times [\stackrel{r}{\ldots}]\times E_d = E_1\times \cdots E_{r-1}\times E_{r+1}\times \cdots \times E_d. $$
Given $A \in {\cal L}(E_1, \ldots, E_d;F)$, for $r=1,\ldots,d$ and $x^r\in E_r$, we define
$$A_{x^r}\colon E_1\times [\stackrel{r}{\ldots}]\times E_d\longrightarrow F~,~ A_{x^r}(x^1,[\stackrel{r}{\ldots}],x^d) = A(x^1,\ldots,x^d).$$
It is plain that $A_{x^r}$ is $(d-1)$-linear, continuous and $\|A_{x^r}\| \leq \|A\|\cdot \|x^r\|$.

A Banach multi-ideal $(\cal M, \|\cdot\|_{\cal M})$ is downward coherent if, 
 regardless of the operator $A \in {\cal M}(E_1, \ldots, E_d;F)$, $r = 1, \ldots, d$, and $x^r \in E_r$, it holds $A_{x^r} \in {\cal M}(E_1, [\stackrel{r}{\ldots}], E_d;F)$ and $\|A_{x^r}\|_{\cal M} \leq \|A\|_{\cal M} \cdot\|x^r\|$. 

\subsection{The isotropic case}

Throughout this subsection 
$\mathbb{N}^d$ is regarded as the block associated to the trivial partition 
$\mathcal{I}_t$ and to the sequences 
$(j_n^r)_{n=1}^{\infty}$, $r=1,\ldots,d$, defined by a fixed bijection 
$\sigma\colon\mathbb{N}\longrightarrow\mathbb{N}^d$, according to Example \ref{exemplo_blocos_c}. We denote $\sigma(n) = (j_n^1,\ldots,j_n^d)$ for every $n\in\mathbb{N}$. To move from $d$-linear operators to $(d-1)$-linear operators we need to regard $\mathbb{N}^{d-1}$ as a block related, in a convenient way, to the sequences that define the block $\mathbb{N}^d$. We do that by considering the set 
\begin{align}
    B = \{(j^1,\ldots,j^d)\in\mathbb{N}^d\colon j^m = 1\},
    \label{conjunto_B}
\end{align}
and defining $n_1 = \min\{ n : \sigma(n)\in B\}$, $n_2  = \min\{ n \neq n_1 : \sigma(n)\in B\}, \ldots$, $n_k = \min\{ n : n\notin \{n_1,\cdots, n_{k-1}\} \mbox{ and } \sigma(n)\in B\}, \ldots$. 
Since $B$ is infinite and $\sigma$ is a bijection, this provides an increasing sequence 
$(n_k)_{k=1}^\infty$ of natural numbers.  
For each $m \in \{1, \ldots, d\}$, consider the subsequences $(j_{n_k}^r)_{k=1}^{\infty}$ of $(j_n^r)_{n=1}^{\infty}$, $r=1,[\stackrel{m}{\ldots}],d$. It is not difficult to check that the correspondence 
$$k\in\mathbb{N}\mapsto (j_{n_k}^1,[\stackrel{m}{\ldots}],j_{n_k}^d)\in\mathbb{N}^{d-1} $$
is a bijection, hence $\mathbb{N}^{d-1}$ is a block associated to the trivial partition and to these subsequences. This is how $\mathbb{N}^{d-1}$ shall be regarded from now on.

\begin{definition}\rm We say that a sequence class $X$ is \textit{0-invariant} if, for every Banach space $E$ and any $E$-valued sequence $(x_j)_{j=1}^{\infty}$, it holds
	\begin{align*}
	    (x_j)_{j=1}^{\infty} \in X(E) \Longleftrightarrow (x_j^0)_{j=1}^{\infty}\in X(E) \mbox{ e } \|(x_j)_{j=1}^{\infty} \|_{X(E)} = \|(x_j^0)_{j=1}^{\infty} \|_{X(E)},
	\end{align*}
where $(x_j^0)_{j=1}^{\infty}$ stands for the zerofree version of $(x_j)_{j=1}^{\infty}$, meaning that $x_j^0$ is the $j$-th nonzero coordinate of $(x_j)_{j=1}^{\infty}$, if it exists, or zero otherwise (see \cite[Definition 3.1]{botelho_complete_2020} or, in the case of scalar sequences,  \cite{mathnotes}).
\end{definition}

Most of the usual sequence classes are $0$-invariant.


\begin{proposition}
     Let $X_1,\ldots,X_d$ and $Y$ be sequence classes with $Y$ being $0$-invariant,  $1\leq m\leq d$, $E_1,\ldots, E_d, F$ be Banach spaces and $a^m\in E_m$.  If  $A\in\mathcal{L}_{X_1,\ldots,X_d;Y}^{\mathbb{N}^d}(E_1,\ldots,E_d;F)$, then $$A_{a^m}\in \mathcal{L}_{X_1,[\stackrel{m}{\ldots}],X_d;Y}^{\mathbb{N}^{d-1}}(E_1,[\stackrel{m}{\ldots}],E_d;F) {\rm ~~and~~} \| A_{a^m}\|_{\mathbb{N}^{d-1};X_1,[\stackrel{m}{\ldots}],X_d;Y}\leq \| A\|_{\mathbb{N}^d;X_1,\ldots,X_d;Y}\cdot \| a^m\|.$$
    \label{prop-decaida-isotropico}
\end{proposition}
\begin{proof} Of course we can suppose $a^m \neq 0$. Let 
$(x_j^r)_{j=1}^{\infty}\in X_r(E_r)$, $r=1,[\stackrel{m}{\ldots}],d$ be given. Define
 $(x_j^m)_{j=1}^{\infty}:= a^m\cdot e_1 = (a^m, 0,0, \ldots) \in X_m(E_m) $. 
 By asumption,  
     $$Y(F) \ni \widehat{A}_{\mathbb{N}^d}\left( (x_j^1)_{j=1}^{\infty},\ldots,(x_j^d)_{j=1}^{\infty}\right) = \left( A(x_{j_n^1}^1,\ldots,x_{j_n^d}^d)\right)_{n=1}^{\infty}= \left( A_{x_{j_n^m}^m}(x_{j_n^1}^1,[\stackrel{m}{\ldots}],x_{j_n^d}^d)\right)_{n=1}^{\infty}.
    $$
Using the set $B$ defined in (\ref{conjunto_B}) and considering the corresponding increasing sequence $(n_k)_{k=1}^{ \infty}$ of natural numbers, note that:\\
 $\bullet$ For every $k\in\mathbb{N}$, $A_{x_{j_{n_k}^m}^m}(x_{j_n^1}^1,[\stackrel{m}{\ldots}],x_{j_n^d}^d) = A_{a^m}(x_{j_n^1}^1,[\stackrel{m}{\ldots}],x_{j_n^d}^d)$ as $j_{n_k}^m = 1$ and $x_1^m = a^m$.\\
$\bullet$ For every $n\in \mathbb{N}\backslash\{n_k: k\in\mathbb{N}\}$, $A_{x_{j_n^m}^m}(x_{j_n^1}^1,[\stackrel{m}{\ldots}],x_{j_n^d}^d) = 0$ as $j_n^m\neq 1$ and $x_j^m = 0$ for every $j\neq 1$.

Thus, definining $y=\left( A_{x_{j_n^m}^m}(x_{j_n^1}^1,[\stackrel{m}{\ldots}],x_{j_n^d}^d)\right)_{n=1}^{\infty}$ and $z =\left( A_{a^m}(x_{j_{n_k}^1}^1,[\stackrel{m}{\ldots}],x_{j_{n_k}^d}^d)\right)_{k=1}^{\infty} $, we have that $z$ is a subsequence of $y$ and all the terms of $y$ not belonging to $z$ are null. It follows that the $n$-th non-null coordinate of $z$ is equal to the $n$-th non-null coordinate of $y$, that is, $y^0 = z^0$. Since $Y$ is $0$-invariant,
    \begin{align*}
       y \in Y(F) \Longleftrightarrow y^0\in Y(F) \Longleftrightarrow z^0\in Y(F) \Longleftrightarrow z\in Y(F) \mbox{~and}
    \end{align*}
        \begin{align*}
       Y(F) \ni   & ~\Big\| \widehat{A}_{\mathbb{N}^d}\left( (x_j^1)_{j=1}^{\infty},\ldots,(x_j^d)_{j=1}^{\infty}\right)\Big\|_{Y(F)} = \| y \|_{Y(F)} = \| y^0 \|_{Y(F)} = \| z^0 \|_{Y(F)} = \| z \|_{Y(F)} \\
        &=\Big\| \left( A_{a^m}(x_{j_{n_k}^1}^1,[\stackrel{m}{\ldots}],x_{j_{n_k}^d}^d)\right)_{k=1}^{\infty}\Big\|_{Y(F)}    = \Big\| \left( A_{a^m}\right)_{\mathbb{N}^{d-1}}^{\widehat{}}\left( (x_j^1)_{j=1}^{\infty},[\stackrel{m}{\ldots}],(x_j^d)_{j=1}^{\infty}\right)\Big\|_{Y(F)}.
    \end{align*}
        This proves that $A_{a^m}\in \mathcal{L}_{X_1,[\stackrel{m}{\ldots}],X_d;Y}^{\mathbb{N}^{d-1}}(E_1,[\stackrel{m}{\ldots}],E_d;F)$ and, replacing $a^m$ with $b := \frac{a^m}{\| a^m\|}$, we get
    \begin{align*}
        &\| \widehat{A}_{\mathbb{N}^d}\|  = \sup_{\tiny{\begin{array}{c}
            (x^r_j)_{j=1}^{\infty}\in B_{X_{r}(E_r)} \nonumber\\
            r=1,\ldots,d
        \end{array}}} \Big\| \widehat{A}_{\mathbb{N}^d}\left( (x_j^1)_{j=1}^{\infty},\ldots,(x_j^d)_{j=1}^{\infty}\right)\Big\|_{Y(F)} \nonumber\\
        &\geq \sup_{\tiny{\begin{array}{c}(x^r_j)_{j=1}^{\infty}\in B_{X_{r}(E_r)} \nonumber\\ r=1,[\stackrel{m}{\ldots}],d \end{array}}} \Big\| \widehat{A}_{\mathbb{N}^d}\left( (x_j^1)_{j=1}^{\infty},\ldots,(x_j^{m-1})_{j=1}^{\infty},b\cdot e_1,(x_j^{m+1})_{j=1}^{\infty},(x_j^d)_{j=1}^{\infty}\right)\Big\|_{Y(F)}  \nonumber\\
        &=  \frac{1}{\| a^m\|}\sup_{\tiny{\begin{array}{c}(x^r_j)_{j=1}^{\infty}\in B_{X_{r}(E_r)} \nonumber\\ r=1,[\stackrel{m}{\ldots}],d \end{array}}} \Big\| \widehat{A}_{\mathbb{N}^d}\left( (x_j^1)_{j=1}^{\infty},\ldots,(x_j^{m-1})_{j=1}^{\infty},a^m\cdot e_1,(x_j^{m+1})_{j=1}^{\infty},(x_j^d)_{j=1}^{\infty}\right)\Big\|_{Y(F)} \nonumber\\
        &=  \frac{1}{\| a^m\|}\sup_{\tiny{\begin{array}{c}(x^r_j)_{j=1}^{\infty}\in B_{X_{r}(E_r)} \nonumber\\ r=1,[\stackrel{m}{\ldots}],d \end{array}}} \Big\| \left( A_{a^m}\right)_{\mathbb{N}^{d-1}}^{\widehat{}}\left( (x_j^1)_{j=1}^{\infty},[\stackrel{m}{\ldots}],(x_j^d)_{j=1}^{\infty}\right)\Big\|_{Y(F)} \nonumber= \frac{\| \left( A_{a^m}\right)_{\mathbb{N}^{d-1}}^{\widehat{}}\|}{\| a^m\|}, \label{eq.23}
    \end{align*}
from which it follows that 
    \begin{align*}
        \| A_{a^m}\|_{\mathbb{N}^{d-1};X_1,[\stackrel{m}{\ldots}],X_d;Y} = \| \left( A_{a^m}\right)_{\mathbb{N}^{d-1}}^{\widehat{}}\| \leq \| \widehat{A}_{\mathbb{N}^d}\|\cdot \| a^m\| = \| A\|_{\mathbb{N}^d;X_1,\ldots,X_d;Y}\cdot \| a^m\|,
    \end{align*}
\end{proof}

\begin{corollary}\label{corisot}
    Let $X_1,\ldots,X_d$ and $Y$ be sequence classes with $Y$ being $0$-invariant. If $\mathcal{L}_{X_1,\ldots,X_d;Y}^{\mathbb{N}^d}(E_1,\ldots,E_d;F) = \mathcal{L}(E_1,\ldots,E_d;F)$, then:\\
{\rm (a) }$\mathcal{L}_{X_1,[\stackrel{m}{\ldots}],X_d;Y}^{\mathbb{N}^{d-1}}(E_1,[\stackrel{m}{\ldots}],E_d;F) = \mathcal{L}(E_1,[\stackrel{m}{\ldots}],E_d;F)$  for every $m=1,\ldots,d$.\\
{\rm (b)} 
$\mathcal{L}_{X_r;Y}(E_r;F) = \mathcal{L}(E_r;F)$ for every $r=1,\ldots,d$.
\end{corollary}
\begin{proof} 
   (a) Given $B\in\mathcal{L}(E_1,[\stackrel{m}{\ldots}],E_d;F)$, choose $y\in E_m$ with $\| y\| = 1$ and take, by the Hahn-Banach Theorem, 
    $\varphi\in E_m^*$ such that $\| \varphi\|= 1$ and $\varphi(y) = \| y\| =1$. Considering the operator $A\in\mathcal{L}(E_1, \ldots, E_d;F)$ given by
    $A( x^1,\ldots,x^d) = B(x^1,[\stackrel{m}{\ldots}],x^d)\cdot \varphi(x^m)$, 
we have $B=A_y$. By assumption, $A\in \mathcal{L}_{X_1,\ldots,X_d;Y}^{\mathbb{N}^d}(E_1,\ldots,E_d;F)$, so Proposition  \ref{prop-decaida-isotropico} yields that $B=A_y \in \mathcal{L}_{X_1,[\stackrel{m}{\ldots}],X_d;Y}^{\mathbb{N}^{d-1}}(E_1,[\stackrel{m}{\ldots}],E_d;F).$ \\
(b) Just apply (a) repeteadly.   
\end{proof}


By $\Pi_{q;p}(E;F)$ we denote the space of absolutely $(q,p)$-summing linear operators from $E$ to $F$. Since the sequence class $\ell_q(\cdot)$ is $0$-invariante, we recover the following result from \cite{pellegrino_fully_2005} as a particular case of the corollary above. 
\begin{corollary}{\rm \cite[Corollary 3.4]{pellegrino_fully_2005}}
	Let $q,p_1, \ldots, p_d \geq 1$ and $E_1, \ldots, E_d, F$ be Banach spaces such that  $\mathcal{L}(E_1,\ldots,E_d;F) = \mathcal{L}_{\ell_{p_1}^w(\cdot),\ldots,\ell_{p_d}^w(\cdot);\ell_q(\cdot)}^{\mathbb{N}^d}(E_1,\ldots,E_d;F)$.  Then $\Pi_{q;p_r}(E_r;F) =\mathcal{L}(E_r;F)$ for every $r=1,\ldots,d$.
\end{corollary}

\subsection{The anisotropic case}
\label{aplicacoes-caso-anisotrópico}
In this subsection we consider $\mathbb{N}^d$ as the block associated to the discrete partition of $\{1, \ldots, d\}$ and to the sequences 
$(j_n^1)_{n=1}^{\infty} = \cdots = (j_n^d)_{n=1}^{\infty} = (n)_{n=1}^{\infty}$. In the same fashion, $\mathbb{N}^{d-1}$ shall be regarded as the block associated to the discrete partition of $\{1,\ldots,d-1\}$ and to the sequences $(j_n^1)_{n=1}^{\infty} = \cdots = (j_n^{d-1})_{n=1}^{\infty} = (n)_{n=1}^{\infty}$.

\begin{proposition}\label{prop_coerencia_anisotropico}
     Let $X_1,\ldots,X_d$ and $Y$ be sequence classes with $Y$ linearly stable, $E_1, \ldots,$ $E_d$, $F$ be     Banach spaces and $a\in E_1$ be given. If $A\in\mathcal{L}_{X_1,\ldots,X_d;^dY}^{\mathbb{N}^d}(E_1,\ldots,E_d;F)$, then $A_a\in \mathcal{L}_{X_2,\ldots,X_d;^{d-1}Y}^{\mathbb{N}^{d-1}}(E_2,\ldots,E_d;F)$ and $$\| A_a\|_{\mathbb{N}^{d-1};X_2,\ldots,X_d;^{d-1}Y} \leq \| A\|_{\mathbb{N}^d;X_1,\ldots,X_d;^dY} \cdot \| a\|.$$
\end{proposition}
\begin{proof} Of course we can suppose $a \neq 0$.
Let the sequences $(x_j^2)_{j=1}^{\infty}\in X_2(E_2),\ldots,$ $(x_j^d)_{j=1}^{\infty}\in X_d(E_d)$ be given. 
Consider the sequence $(x_j^1)_{j=1}^{\infty} := a\cdot e_1  = (a, 0, 0, \ldots)\in X_1(E_1)$. 
By the definition of  $\mathcal{L}_{X_1,\ldots,X_d;^dY}^{\mathbb{N}^d}(E_1,\ldots,E_d;F)$,
    \begin{align*}
        Y(\stackrel{d}{\cdots} Y(F)\cdots)&\ni \widehat{A}_{\mathbb{N}^d} \left( (x_j^1)_{j=1}^{\infty},\ldots,(x_j^d)_{j=1}^{\infty}\right) = \left(\cdots\left( A(x_{n_1}^1,\ldots,x_{n_d}^d)\right)_{n_d=1}^{\infty}\cdots\right)_{n_1=1}^{\infty} \\
        &= \left(\left(\cdots\left( A(a, x_{n_2}^2,\ldots,x_{n_d}^d) \right)_{n_d=1}^{\infty}\cdots\right)_{n_2=1}^{\infty},0,0,0,\ldots\right)\\
        &= \left(\left(\cdots\left( A_{a}(x_{n_2}^2,\ldots,x_{n_d}^d) \right)_{n_d=1}^{\infty}\cdots\right)_{n_2=1}^{\infty},0,0,0,\ldots\right)\\
        &= \left(\cdots\left( A_{a}(x_{n_2}^2,\ldots,x_{n_d}^d) \right)_{n_d=1}^{\infty}\cdots\right)_{n_2=1}^{\infty}\cdot e_1.
    \end{align*}
Since the coordinates of a sequence in $Y(\stackrel{d}{\cdots} Y(F)\cdots)$ belong to $Y(\stackrel{d-1}{\cdots} Y(F)\cdots)$, we have
    \begin{align*}
        \left(\cdots\left( A_a(x_{n_2}^2,\cdots,x_{n_d}^d) \right)_{n_d=1}^{\infty}\ldots\right)_{n_2=1}^{\infty}\in Y(\stackrel{d-1}{\cdots} Y(F)\cdots),
    \end{align*}
which gives $A_a\in \mathcal{L}_{X_2,\ldots,X_d;^{d-1}Y}^{\mathbb{N}^{d-1}}(E_2,\ldots,E_d;F)$. By \cite[Lemma 2.4]{{botelho_transformation_2017}}, $\| y\cdot e_1 \|_{Y(F)} = \| y\|_F$ for every $y\in F$, therefore
    \begin{align*}
        \Big\| \widehat{A}_{\mathbb{N}^d}& \left( a\cdot e_1, (x_j^2)_{j=1}^{\infty},\ldots,(x_j^d)_{j=1}^{\infty}\right)\Big\| = \Big\| \left(\cdots\left( A_{a}(x_{n_2}^2,\cdots,x_{n_d}^d) \right)_{n_d=1}^{\infty}\ldots\right)_{n_2=1}^{\infty}\cdot e_1\Big\|  \\
        &= \Big\| \left(\cdots\left( A_a(x_{n_2}^2,\ldots,x_{n_d}^d) \right)_{n_d=1}^{\infty}\cdots\right)_{n_2=1}^{\infty}\Big\|= \Big\| \left( A_a\right)_{\mathbb{N}^{d-1}}^{\widehat{}}\left( (x_j^2)_{j=1}^{\infty},\ldots,(x_j^d)_{j=1}^{\infty}\right)\Big\|
    \end{align*}
for all $(x_j^r)_{j=1}^{\infty}\in X_r(E_r)$, $r\in\{2,\ldots,d\}$. 
Applying this equality to $\frac{a}{\| a\|}$, we get     \begin{align*}
    \| A\|_{\mathbb{N}^d;X_1,\ldots,X_d;Y}&=     \| \widehat{A}_{\mathbb{N}^d}\|= \sup_{\tiny{\begin{array}{c} (x_j^r)_{j=1}^{\infty}\in B_{X_r(E_r)}\\ r\in \{1,\ldots,d\}        \end{array}}} \Big\| \widehat{A}_{\mathbb{N}^d}\left( (x_j^1)_{j=1}^{\infty},\ldots,(x_j^d)_{j=1}^{\infty}\right)\Big\|\\
        &\geq \sup_{\tiny{\begin{array}{c} (x_j^r)_{j=1}^{\infty}\in B_{X_r(E_r)}\\ r\in \{2,\ldots,d\}        \end{array}}} \Big\| \widehat{A}_{\mathbb{N}^d}\left( \frac{a}{\| a\|}\cdot e_1,(x_j^2)_{j=1}^{\infty},\ldots,(x_j^d)_{j=1}^{\infty}\right)\Big\| \\
        &= \frac{1}{\| a\|}\sup_{\tiny{\begin{array}{c} (x_j^r)_{j=1}^{\infty}\in B_{X_r(E_r)}\\ r\in \{2,\ldots,d\}        \end{array}}} \Big\| \widehat{A}_{\mathbb{N}^d}\left( a\cdot e_1,(x_j^2)_{j=1}^{\infty},\ldots,(x_j^d)_{j=1}^{\infty}\right)\Big\| \\
         &= \frac{1}{\| a\|}\sup_{\tiny{\begin{array}{c} (x_j^r)_{j=1}^{\infty}\in B_{X_r(E_r)}\\ r\in \{2,\ldots,d\}        \end{array}}} \Big\| \left( A_a\right)_{\mathbb{N}^{d-1}}^{\widehat{}}\left( (x_j^2)_{j=1}^{\infty},\ldots,(x_j^d)_{j=1}^{\infty}\right)\Big\| \\
         &= \frac{\| \left( A_a\right)_{\mathbb{N}^{d-1}}^{\widehat{}}\|  }{\| a\|} = \frac{\| A_a\|_{\mathbb{N}^{d-1};X_2,\ldots,X_d;Y}}{\|a\|}.
    \end{align*}
\end{proof}

Coherence is usually applied for classes of polynomials ideals, which are closely connected to symmetric multilinear operators. Let us see that, in the result above, if the operator is symmetric then we can fix any coordinate, and not only the first one. We assume that symmetric multilinear operators are familiar to the reader.



For $A\in\mathcal{L}(^dE;F)$, $a \in E$ and $m \in \{1, \ldots, d\}$, by $A_a^m$ we mean the operator
\begin{align*}
    A_a^m \colon E^{d-1}\longrightarrow F~,~ A_a^m(x^1,[\stackrel{m}{\ldots}],x^d) = A(x^1,\ldots, x^{m-1},a,x^{m+1},\ldots,x^d).
\end{align*}

\begin{corollary}
     Let $X$ and $Y$ be sequence with $Y$ linearly stable, $a\in E$ and $m\in\{1,\ldots,d\}$. If $A\in \mathcal{L}_{^dX;^dY}^{\mathbb{N}^d}(^dE;F)$ is a symmetric operator, then $A_a^m\in \mathcal{L}_{^{d-1}X;^{d-1}Y}^{\mathbb{N}^{d-1}}(^{d-1}E;F)$ and
    $$\| A_a^m\|_{\mathbb{N}^{d-1};^{d-1}X;^{d-1}Y} \leq \| A\|_{\mathbb{N}^d;^dX;^dY} \cdot \| a\|.$$
    \end{corollary}

\begin{proof} Considering the permutation $\alpha \colon \{1,\ldots,d\}\longrightarrow \{1,\ldots,d\}$ given by
    \begin{align*}
        \alpha(r) := \begin{cases} m&, \mbox{ if } r=1\\ r-1&, \mbox{ if } r=2,\ldots,m \\ r&, \mbox{ if } r\geq m,
    \end{cases}
    \end{align*}
the symmetry of $A$ implies that
$A_a^m = A_a^1$. The result follows from Proposition  \ref{prop_coerencia_anisotropico}.
\end{proof}

We finish the paper with some coincidence results, which follow from Proposition \ref{prop_coerencia_anisotropico} reasoning similarly to the proof of Corollary \ref{corisot}.

\begin{corollary}
    Let $X_1,\ldots,X_d$ and $Y$ be sequence classes with $Y$ linearly stable and $E_1, \ldots, E_d, F$ be Banach spaces. If $\mathcal{L}(E_1,\ldots,E_d;F) = \mathcal{L}_{X_1,\ldots,X_d;^dY}^{\mathbb{N}^d}(E_1,\ldots,E_d;F)$, then $\mathcal{L}(E_2,\ldots,E_d;F) = \mathcal{L}_{X_2,\ldots,X_d;^{d-1}Y}^{\mathbb{N}^{d-1}}(E_2,\ldots,E_d;F)$. 
\end{corollary}

\begin{corollary}
    Let $X, Y$ be sequence classes with $Y$ linearly stable and $E,F$ be Banach spaces. If $\mathcal{L}(^dE;F) = \mathcal{L}_{^dX;^dY}^{\mathbb{N}^d}(^dE;F)$ for some $d \in \mathbb{N}$, then $\mathcal{L}(E;F) = \mathcal{L}_{X;Y}(E;F)$.
\end{corollary}

\bigskip

\noindent Faculdade de Matem\'atica~~~~~~~~~~~~~~~~~~~~~~Departamento de Matem\'atica\\
Universidade Federal de Uberl\^andia~~~~~~~~ IMECC-UNICAMP\\
38.400-902 -- Uberl\^andia -- Brazil~~~~~~~~~~~~ 13.083-859 -- Campinas -- Brazil\\
e-mail: botelho@ufu.br ~~~~~~~~~~~~~~~~~~~~~~~~~e-mail: davidson.freitas@ifgoiano.edu.br


\begin{thebibliography}{99}\footnotesize

\vspace*{-0.5em}
    \bibitem{achour_factorization_2014} Achour, D., Dahia, E., Rueda, P. and Sánchez Pérez, E. A. {\it Factorization of strongly $(p,\sigma)$-continuous multilinear operators}. Linear Multilinear Algebra {\bf 62} (2014), no. 12, 1649–1670.

\vspace*{-0.5em}
    \bibitem{achour_cohen_2007} Achour, D. and Mezrag, L. {\it On the Cohen strongly $p$-summing multilinear operators}. J.  Math. Anal. Appl. {\bf 327} (2007), no. 1, 550–563.

\vspace*{-0.5em}
    \bibitem{albuquerque_summability_2018} Albuquerque, N., Araújo, G., Cavalcante, W., Nogueira, T., Núñez, D., Pellegrino, D. and Rueda, P. {\it On summability of multilinear operators and applications.} Ann. Funct. Anal. {\bf 9} (2018), no. 4, 574–590.


\vspace*{-0.5em}
    \bibitem{albuquerque_note_2019} Albuquerque, N., Araújo, G., Pellegrino, D. and Rueda, P. {\it A note on multiple summing operators and applications.} Linear Multilinear Algebra  {\bf 67} (2019), no. 4, 660–671.

\vspace*{-0.5em}
    \bibitem{albuquerque_holders_2017a} Albuquerque, N., Araújo, G., Pellegrino, D. and Seoane-Sepúlveda, J. B. {\it Hölder’s inequality: some recent and unexpected applications}.  Bull. Belg. Math. Soc. Simon Stevin {\bf 24} (2017), no. 2, 199–225.

\vspace*{-0.5em}
    \bibitem{albuquerque_summability_2019} Albuquerque, N., Araújo, G., Rezende, L. and Santos, J. {\it A summability principle and applications}. arXiv:1904.04549 [math], 2019.

\vspace*{-0.5em}
    \bibitem{albuquerque_sharp_2014} Albuquerque, N., Bayart, F., Pellegrino, D.  and Seoane-Sepúlveda, J. B. {\it Sharp generalizations of the multilinear Bohnenblust-Hille inequality}. J. Funct. Anal. {\bf 266} (2014), no. 6, 3726–3740.

\vspace*{-0.5em}
    \bibitem{albuquerque_optimal_2016} Albuquerque, N., Bayart, F., Pellegrino, D. and Seoane-Sepúlveda, J. B. {\it Optimal Hardy-Littlewood type inequalities for polynomials and multilinear operators}. Israel J. Math. {\bf  211}  (2016), no. 1, 197–220.

\vspace*{-0.5em}
    \bibitem{albuquerque_applications_2017} Albuquerque, N., Nogueira, T., Núñez-Alarcón, D., Pellegrino, D. and Rueda, P. {\it Some applications of the Hölder inequality for mixed sums.} Positivity {\bf 21} (2017), no. 4, 1575–1592.

\vspace*{-0.5em}
    \bibitem{albuquerque_absolutely_2015} Albuquerque, N., Núñez-Alarcón, D., Santos, J. and Serrano-Rodríguez, D. M. {\it Absolutely summing multilinear operators via interpolation.} J. Funct. Anal. {\bf 269} (2015), no. 6, 1636–1651.

\vspace*{-0.5em}
    \bibitem{albuquerque_anisotropic_2018} Albuquerque, N. and Rezende, L. {\it Anisotropic regularity principle in sequence spaces and applications.} Commun. Contemp. Math. {\bf 20} (2018), no. 7, 1750087, 14 pp.

\vspace*{-0.5em}
    \bibitem{alencar_some_1989} Alencar, R. and Matos, M. { \it Some classes of multilinear mappings between Banach spaces.} Publicaciones Departamento Análisis Matemático, Unviversidad Complutense Madrid, 12 (1989).

\vspace*{-0.5em}
    \bibitem{araujo_classical_2016} Araújo, G. {\it Some classical inequalities, summability of multilinear operators and strange functions.} Doctoral Thesis, UFPB, João Pessoa, 2016.

\vspace*{-0.5em}
    \bibitem{araujo_optimal_2015} Araújo, G. and Pellegrino, D. {\it Optimal Hardy-Littlewood type inequalities for $m$-linear forms on $\ell_p$ spaces with $1\leq p\leq m$.} Arch. Math. (Basel) 105 (2015), no. 3, 285–295.

\vspace*{-0.5em}
    \bibitem{aron_optimal_2017} Aron, R. M., Núñez-Alarcón, D., Pellegrino, D. M. and Serrano-Rodríguez, D. M. {\it Optimal exponents for Hardy-Littlewood inequalities for $m$-linear operators.} Linear Algebra Appl. {\bf  531} (2017), 399–422.

\vspace*{-0.5em}
    \bibitem{bayart_multiple_2018} Bayart, F. {\it Multiple summing maps: coordinatewise summability, inclusion theorems and $p$-Sidon sets.} J. Funct. Anal. {\bf 274} (2018), no. 4, 1129–1154.

\vspace*{-0.5em}
    \bibitem{bayart_coincidence_2020} Bayart, F., Pellegrino, D. and Rueda, P. {\it On coincidence results for summing multilinear operators: interpolation, $\ell_1$-spaces and cotype}. Collect. Math. {\bf 71} (2020), no. 2, 301–318.



\vspace*{-0.5em}
    \bibitem{blasco_coincidence_2016} Blasco, O., Botelho, G., Pellegrino, D. and Rueda, P. {\it Coincidence results for summing multilinear mappings.} Proc. Edinb. Math. Soc. (2) {\bf 59} (2016), no. 4, 877–897.

\vspace*{-0.5em}
    \bibitem{bombal_multilinear_2004} Bombal, F., Pérez-García, D. and Villanueva, I. {\it Multilinear extensions of Grothendieck’s theorem.} Q. J. Math. {\bf 55} (2004), no. 4, 441–450.

\vspace*{-0.5em}
    \bibitem{botelho_almost_2001} Botelho, G., Braunss, H.-A. and Junek, H. {\it Almost $p$-summing polynomials and multilinear mappings.} Arch. Math. (Basel) {\bf 76} (2001), no. 2, 109–118.

\vspace*{-0.5em}
     \bibitem{botelho_holomorphy_2006} Botelho, G., Braunss, H.-A., Junek, H. and Pellegrino, D. {\it Holomorphy types and ideals of multilinear mappings.} Studia Math. {\bf 177} (2006), no. 1, 43–65.

\vspace*{-0.5em}
    \bibitem{botelho_inclusions_2009} Botelho, G., Braunss, H.-A., Junek, H. and Pellegrino, D. {\it Inclusions and coincidences for multiple summing multilinear mappings}. Proc. Amer. Math. Soc. {\bf 137} (2009), no. 3, 991–1000.



\vspace*{-0.5em}
    \bibitem{botelho_type_2016} Botelho, G. and Campos, J. R. {\it Type and cotype of multilinear operators}. Rev. Mat. Complut. {\bf 29} (2016), no. 3, 659–676.

\vspace*{-0.5em}
    \bibitem{botelho_transformation_2017} Botelho, G. and Campos, J. R. {\it On the transformation of vector-valued sequences by linear and multilinear operators.} Monatsh. Math. {\bf 183} (2017), no. 3, 415–435.

\vspace*{-0.5em}
    \bibitem{botelho_operator_2017} Botelho, G., Campos, J. R. and Santos, J. {\it Operator ideals related to absolutely summing and Cohen strongly summing operators.} Pacific J. Math. {\bf 287} (2017), no. 1, 1–17.


\vspace*{-0.5em}
    \bibitem{botelho_summing_2020} Botelho, G. and Freitas, D. {\it Summing multilinear operators by blocks: the isotropic and anisotropic cases.} J. Math. Anal. Appl. {\bf 490} (2020), no. 1, 124203, 21 p.

\vspace*{-0.5em}
    \bibitem{botelho_complete_2020} Botelho, G. and Luiz, J. L. P. {\it Complete latticeability in vector-valued sequence spaces}, Math. Nachr, to appear. Available at arXiv:2004.01604 [math],  2020.






\vspace*{-0.5em}
    \bibitem{botelho_when_2009} Botelho, G. and Pellegrino, D. {\it When every multilinear mapping is multiple summing}. Math. Nachr. {\bf 282} (2009), no. 10, 1414–1422.


\vspace*{-0.5em}
    \bibitem{botelho_summability_2008} Botelho, G., Pellegrino, D. and Rueda, P. {\it Summability and estimates for polynomials and multilinear mappings.} Indag. Math. (N. S.) {\bf 19} (2008), no. 1, 23–31.









\vspace*{-0.5em}
    \bibitem{campos_cohen_2014} Campos, J. R. {\it Cohen and multiple Cohen strongly summing multilinear operators.} Linear Multilinear Algebra {\bf 62} (2014), no. 3, 322–346.

\vspace*{-0.5em}
    \bibitem{joedsonjamilson} Campos, J. R. and Santos, J., {\it An anisotropic approach to mid summable sequences}. Colloq. Math. {\bf 161} (2020), no. 1, 35–49.

\vspace*{-0.5em}
    \bibitem{blas20201} Caraballo, B. M. and Fávaro, V. V. {\it Chaos for convolution operators on the space of entire functions of infinitely many complex variables}. Bull. Soc. Math. France {\bf 148} (2020), no. 2, 237–251.

\vspace*{-0.5em}
    \bibitem{blas20202} Caraballo, B. M. and Fávaro, V. V. {\it
Strongly mixing convolution operators on Fréchet spaces of entire functions of a given type and order}. Integral Equations Operator Theory {\bf 92} (2020), no. 4, Paper No. 31, 27 pp.

\vspace*{-0.5em}
    \bibitem{carando_coherent_2009} Carando, D., Dimant, V. and Muro, S. {\it Coherent sequences of polynomial ideals on Banach spaces}. Math. Nachr. {\bf 282} (2009), no. 8, 1111–1133.

\vspace*{-0.5em}
    \bibitem{carando_every_2012} Carando, D., Dimant, V. and Muro, S. {\it Every Banach ideal of polynomials iscompatible with an operator ideal}. Monatsh. Math. {\bf 165} (2012), no. 1, 1–14.

\vspace*{-0.5em}
    \bibitem{carando_holomorphic_2012} Carando, D., Dimant, V. and Muro, S. {\it Holomorphic functions and polynomial ideals on Banach spaces.} Collect.  Math. {\bf 63} (2012), no. 1, 71–91.

\vspace*{-0.5em}
    \bibitem{mathnotes} Carando, D., Mazzitelli, M. and Sevilla-Peris, P. {\it A note on the symmetry of sequence spaces}, Math. Notes, to appear (disponível em arXiv:1905.11621, 2019).



\vspace*{-0.5em}
    \bibitem{defant_coordinatewise_2010} Defant, A., Popa, D. and Schwarting, U. {\it Coordinatewise multiple summing operators in Banach spaces.} J.  Funct. Anal. {\bf 259} (2010), no. 1, 220–242.



\vspace*{-0.5em}
    \bibitem{dimant_diagonal_2019} Dimant, V. and Villafañe, R. {\it Diagonal multilinear operators on Köthe sequence spaces.} Linear Multilinear Algebra {\bf 67} (2019), no. 2, 248–266.

\vspace*{-0.5em}
    \bibitem{dineen_complex_1999} Dineen, S. {\it Complex analysis on infinite-dimensional spaces.} Springer Monographs in Mathematics. Springer-Verlag London, Ltd., London, 1999.

\vspace*{-0.5em}
    \bibitem{fernandez_closed_1996} Fernández, C. S. {\it The closed graph theorem for multilinear mappings.} Internat. J. Math. Math. Sci. {\bf 19} (1996), no. 2, 407–408.

\vspace*{-0.5em}
     \bibitem{fg} Floret, K., and García, D. {\it On ideals of polynomials and multilinear mappings between Banach spaces}. Arch. Math. (Basel) {\bf 81}  (2003), 300–308.




\vspace*{-0.5em}
\bibitem{karnsinha} Karn, A. K. and Sinha, D. P. {\it An operator summability of sequences in Banach spaces}, Glasg. Math. J. {\bf 56} (2014), 427--437.

\vspace*{-0.5em}
    \bibitem{kim_multiple_2007} Kim, S. G. {\it Multiple weakly summing multilinear mappings and polynomials.} Kyungpook Math. J. {\bf 47} (2007), no. 4, 501–517.


\vspace*{-0.5em}
    \bibitem{matos_fully_2003} Matos, M. C. {\it Fully absolutely summing and Hilbert-Schmidt multilinear mappings.} Collect. Math. {\bf 54} (2003), no. 2, 111–136.



\vspace*{-0.5em}
    \bibitem{mezrag_inclusion_2009} Mezrag, L. and Saadi, K. {\it Inclusion theorems for Cohen strongly summing multilinear operators}. Bull. Belg. Math. Soc. Simon Stevin {\bf 16} (2009), no. 1, 1--11.

\vspace*{-0.5em}
    \bibitem{montanaro_applications_2012} Montanaro, A. {\it Some applications of hypercontractive inequalities in quantum information theory.} J.  Math. Phys. {\bf 53} (2012), no. 12, 122206, 15 pp.

\vspace*{-0.5em}
    \bibitem{mujica_complex_1986} Mujica, J. {\it Complex analysis in Banach spaces,} Dover Publications, 2010.

\vspace*{-0.5em}
    \bibitem{muro2017} Muro, S., Pinasco, D. and Savransky, M. {\it Hypercyclic behavior of some non-convolution operators on $H(\mathbb{C}^{\mathbb{N}})$}. J. Operator Theory {\bf 77} (2017), no. 1, 39–59.


\vspace*{-0.5em}
    \bibitem{muro2018} Muro, S., Pinasco, D. and Savransky, M. {\it Dynamics of non-convolution operators and holomorphy types}. J. Math. Anal. Appl. {\bf 468} (2018), 622--641

\vspace*{-0.5em}
    \bibitem{nunez-alarcon_sharp_2019} Núñez-Alarcón, D., Pellegrino, D. and Serrano-Rodríguez, D. M. {\it Sharp anisotropic Hardy-Littlewood inequality for positive multilinear forms.} Results Math. {\bf 74} (2019), no. 4, Paper No. 193, 10 pp.

\vspace*{-0.5em}
    \bibitem{pellegrino_almost_2012} Pellegrino, D. and Ribeiro, J. {\it On almost summing polynomials and multilinear mappings}. Linear  Multilinear Algebra {\bf 60} (2012), no. 6, 397–413.

\vspace*{-0.5em}
    \bibitem{pellegrino_multi-ideals_2014} Pellegrino, D. and Ribeiro, J. {\it On multi-ideals and polynomial ideals of Banach spaces: a new approach to coherence and compatibility.} Monatsh. Math. {\bf 173} (2014), no. 3, 379–415.




\vspace*{-0.5em}
    \bibitem{pellegrino_fully_2005} Pellegrino, D. and Souza, M. {\it Fully summing multilinear and holomorphic mappings into Hilbert spaces.} Math. Nachr. {\bf 278} (2005), no. 7-8, 877–887.

\vspace*{-0.5em}
    \bibitem{pellegrino_fully_2006} Pellegrino, D. M. and Souza, M. L. V. {\it Fully and strongly almost summing multilinear mappings.}  Rocky Mountain J. Math. {\bf 36} (2006), no. 2, 683–698.



\vspace*{-0.5em}
    \bibitem{perez-garcia_unbounded_2008} Pérez-García, D., Wolf, M. M., Palazuelos, C., Villanueva, I. and Junge, M. {\it Unbounded violation of tripartite Bell inequalities.} Comm. Math. Phys. {\bf 279} (2008), no. 2, 455–486.

\vspace*{-0.5em}
    \bibitem{pietsch_operator_1978} Pietsch, A. {\it Operator Ideals}, North-Holland Publishing Company, 1980.

\vspace*{-0.5em}
    \bibitem{pietsch_ideals_1984} Pietsch, A. {\it Ideals of multilinear functionals (designs of a theory).} In Proceedings of the second international conference on operator algebras, ideals, and their applications in theoretical physics (Leipzig, 1983)(1984), vol. 67 of Teubner-Texte Math., pp. 185–199.

\vspace*{-0.5em}
    \bibitem{popa_mixing_2012} Popa, D. {\it Mixing multilinear operators.} Illinois J. Math. {\bf 56} (2012), no. 3, 885–903.

\vspace*{-0.5em}
    \bibitem{popa_composition_2014} Popa, D. {\it Composition results for strongly summing and dominated multilinear operators}. Open Math. {\bf 12} (2014), no. 10, 1433–1446.

\vspace*{-0.5em}		
    \bibitem{popa_operators_2020} Popa, D. {\it Operators with the Maurey–Pietsch multiple splitting property}. Linear Multilinear Algebra {\bf 68} (2020), no. 6, 1201–1217.




\vspace*{-0.5em}
    \bibitem{ribeiro_generalized_2019} Ribeiro, J. and Santos, F. {\it Generalized multiple summing multilinear operators on Banach spaces.} Mediterr. J. Math. {\bf 16} (2019), no. 5, Paper No. 108, 20.



\vspace*{-0.5em}
    \bibitem{ribeiro_absolutely_2021} Ribeiro, J. and Santos, F. {\it Absolutely summing polynomials}. Methods Funct. Anal. Topology {\bf 27} (2021), 74–87. 

\vspace*{-0.5em}
    \bibitem{ribeiro_coherence_2020} Ribeiro, J., Santos, F. and Torres, E. R. {\it Coherence and compatibility: a stronger approach}. Linear Multilinear Algebra, to appear.







\vspace*{-0.5em}
    \bibitem{zalduendo_estimate_1993} Zalduendo, I. {\it An estimate for multilinear forms on $l^p$-spaces.}  Proc. Roy. Irish Acad. Sect. A {\bf 93} (1993), no. 1, 137–142.

\end{thebibliography}
\end{document}